\documentclass{amsart} 

\newtheorem{theorem}{Theorem}[section]

\newtheorem{lemma}{Lemma}[section]

\theoremstyle{definition}

\newtheorem{example}{Example}[section]

\newtheorem{algorithm}{Algorithm}[section]

\usepackage{amscd}
\usepackage{amsthm}
\usepackage{amsfonts}
\usepackage{amsbsy}
\usepackage{amsmath}
\usepackage{amssymb}
\usepackage{color}	
\usepackage{enumerate}
\usepackage{epsfig}
\usepackage[mathscr]{eucal}
\usepackage{fancyhdr}
\usepackage{mathrsfs}
\usepackage{graphicx}
\usepackage{hyperref}
\hypersetup{colorlinks=false, urlcolor=black}
\usepackage{latexsym}
\usepackage{listings}
\usepackage{longtable}
\usepackage{nomencl}
\usepackage[numbers]{natbib}
\usepackage{psfrag}
\usepackage{setspace}
\usepackage{tikz}
\usepackage{url}
\usepackage{verbatim}
\usepackage{graphics,epstopdf}
\usepackage[all]{xy}
\newcommand{\Z}{\mathbb{Z}}
\newcommand{\Q}{\mathbb{Q}}

\newcommand{\K}{\mathbb{K}}
\newcommand{\C}{\mathbb{C}}

\newcommand{\EE}{\mathbb{E}}

\begin{document}

\title[On parameters of orders of quartic fields and essential pairs]{On parameters of orders of quartic fields and essential pairs -- extended version}

\author[Samuel A. Hambleton, Randy Yee]{Samuel A. Hambleton, Randy Yee}

\address{School of Mathematics and Physics, The University of Queensland, St. Lucia, Queensland, Australia 4072}

\email{sah@maths.uq.edu.au}

\address{Department of Computer Science, University of Calgary, 2500 University Drive NW, Calgary, AB, Canada T2N 1N4}

\email{randy.yee1@ucalgary.ca}


\date{February 18, 2020.}


\maketitle

\begin{abstract}
Classes of pairs of ternary quadratic forms parametrize quartic rings by a result of Bhargava. We give an algorithm for finding a pair of ternary quadratic forms that parametrize a given order of a quartic field. We examine a new technique, essential pairs, for obtaining parameters of orders of quartic fields from a number field database. Essential pairs for maximal orders of quartic fields are very common and provide a simple means of obtaining an integral basis for the ring of integers. 
\end{abstract}

\section{Introduction}\label{intro}

Binary cubic forms play an important role in working with cubic number fields. For each cubic number field $\K = \Q (\delta )$ of discriminant $\Delta $, there is a $\text{GL}_2(\Z )$-class of binary cubic forms such that each member $$\mathcal{C}(x, y) = a x^3 + b x^2 y + c x y^2 + d y^3 $$ of the class has discriminant $\Delta $. In defining $\K$ by a form, we may assume with no loss of generality that $\mathcal{C}$ is reduced and irreducible in $\Q [x, y]$ \cite[pp. 125--137]{cfg4} and hence associate the maximal order of $\K$ with a unique reduced binary cubic form $\mathcal{C}$ and choose the generator $\delta $ of $\K$ to be a real root of the cubic polynomial $\mathcal{C}(x, 1)$. This means that there is an integral basis $\left\{ 1, a \delta , a \delta^2 + b \delta \right\}$ found by Nakagawa \cite{Nakagawa4} and Belabas \cite{Belabas4}, where $a$ and $b$ satisfy certain bounds. Levi \cite{Levi4} showed that for each order of a cubic field there is a $\text{GL}_2(\Z )$-class of binary cubic forms that parametrize the order, a relationship further investigated by Delone and Faddeev \cite{DeloneFaddeev4}; binary cubic forms parametrize cubic rings. This allows us to use reduction of binary cubic forms to distinguish between orders of cubic fields and to work with cubic fields with small coefficients $a, b$. Such bounds can be very useful. Hambleton and Williams \cite{cfg4} made use of the bounds on the coefficients of a reduced binary cubic form in their observations on Voronoi's continued fraction algorithm for finding units of cubic fields. This research was motivated by our desire to do something similar with the parameters of orders of quartic fields. However, the problem is complicated by the fact that totally real quartic forms are parametrized by pairs of indefinite ternary quadratic forms and techniques for their reduction are not suitable. Instead some other useful techniques for working with quartic fields were found: the use of matrices rather than multiplication tables, the concept of essential pairs, and the algorithm for finding parameters that produce a given quartic order. 

Bhargava \cite{Bhargavathesis, Bhargavacomp4} showed that classes of pairs of ternary quadratic forms parametrize quartic rings, however the question of how to produce parameters from a given quartic ring, although implicitly given in Bhargava's work \cite{Bhargavacomp4}, was not entirely detailed. O'Dorney \cite{ODorney} has answered this more explicitly, however the results presented are given in terms of ideals and no explicit algorithm or simple example is apparent in that article. We will solve this problem here using the Hermite normal form of an integer matrix and that which we present has the advantage of allowing data from the multiplication table of the quartic ring in question to be directly input into an integer matrix. Our algorithm provides twelve integer parameters that are coefficients of a pair of ternary quadratic forms that parametrize the given order of a quartic field and all calculations are entirely elementary. 

Bhargava's technique \cite{Bhargavacomp4} for understanding quartic rings involved the consideration of a pair of classes of ternary quadratic forms as the objects which parametrize quartic rings. A ternary quadratic form 
\begin{equation*}
\mathcal{Q}_A = a_{11} x^2 + a_{12} x y + a_{13} x z + a_{22} y^2 + a_{23} y z + a_{33} z^2
\end{equation*}
can be expressed using a symmetric half-integral matrix $A$, where 
\begin{align*}
A & = \frac{1}{2} \left(
\begin{array}{ccc}
 2 a_{11} & a_{12} & a_{13} \\
 a_{12} & 2 a_{22} & a_{23} \\
 a_{13} & a_{23} & 2 a_{33} \\
\end{array}
\right) , & \mathcal{Q}_A & = \left(
\begin{array}{ccc}
 x & y & z \\
\end{array}
\right) A \left(
\begin{array}{c}
 x \\
 y \\
 z \\
\end{array}
\right) ,
\end{align*}
and $a_{ij} \in \Z$. Classes of pairs of ternary quadratic forms are defined in terms of matrix multiplication involving elements of $\text{GL}_3(\Z )$ and $\text{GL}_2(\Z )$. The pair of ternary quadratic forms $\left( \mathcal{Q}_A , \mathcal{Q}_B \right)$ corresponds to the pair of matrices $(A, B)$. The action of $G_3 \in \text{GL}_3(\Z )$ on $(A, B)$ is to send $(A, B)$ to the pair of matrices $\left( G_3^{\top} A G_3 , G_3^{\top} B G_3 \right)$, where $G_3^{\top}$ denotes the transpose of $G_3$. A matrix $\left(
\begin{array}{cc}
 p & q \\
 r & s \\
\end{array}
\right) \in \text{GL}_2(\Z )$ also has an action on $(A, B)$, which is to send $(A, B)$ to the pair of matrices $(p A + q B, r A + s B)$. We elaborate on these two actions below.

Let 
\begin{eqnarray*}
\mathcal{Q}_A & = & a_{11} x^2 + a_{12} x y + a_{13} x z + a_{22} y^2 + a_{23} y z + a_{33} z^2 , \\
\mathcal{Q}_B & = & b_{11} x^2 + b_{12} x y + b_{13} x z + b_{22} y^2 + b_{23} y z + b_{33} z^2 
\end{eqnarray*}
be the ternary quadratic forms corresponding to the symmetric half-integral matrices 
\begin{align*}
A & = \frac{1}{2} \left(
\begin{array}{ccc}
 2 a_{11} & a_{12} & a_{13} \\
 a_{12} & 2 a_{22} & a_{23} \\
 a_{13} & a_{23} & 2 a_{33} \\
\end{array}
\right) , & B & = \frac{1}{2} \left(
\begin{array}{ccc}
 2 b_{11} & b_{12} & b_{13} \\
 b_{12} & 2 b_{22} & b_{23} \\
 b_{13} & b_{23} & 2 b_{33} \\
\end{array}
\right) , 
\end{align*}
and let 
\begin{equation*}
G_3 = \left(
\begin{array}{ccc}
 u_{11} & u_{12} & u_{13} \\
 u_{21} & u_{22} & u_{23} \\
 u_{31} & u_{32} & u_{33} \\
\end{array}
\right) \in \text{GL}_3(\Z ) .
\end{equation*}
Then the replacement (denoted by $\mathcal{Q}_A \circ G_3$) 
\begin{equation*}
\left(
\begin{array}{c}
 x \\
 y \\
 z \\
\end{array}
\right) \longmapsto G_3 \left(
\begin{array}{c}
 x \\
 y \\
 z \\
\end{array}
\right) 
\end{equation*}
in $\mathcal{Q}_A (x, y, z)$ and $\mathcal{Q}_B (x, y, z)$ produces a pair of ternary quadratic forms \\ 
$\left( \mathcal{Q}_{\overline{A}} (x, y, z), \mathcal{Q}_{\overline{B}} (x, y, z) \right)$ that corresponds to the pair of matrices $$\left( \overline{A} , \overline{B} \right) = \left( G_3^{\top} A G_3 , G_3^{\top} B G_3 \right) .$$ Notice that 
\begin{equation*}
\mathcal{Q}_A \circ G_3 = \left(
\begin{array}{ccc}
 x & y & z \\
\end{array}
\right) G_3^{\top} A G_3 \left(
\begin{array}{c}
 x \\
 y \\
 z \\
\end{array}
\right) .
\end{equation*}

We have described the action of an element $G_3 \in \text{GL}_3(\Z )$ on $(A, B)$. Now consider the action of $G_2 \in \text{GL}_2(\Z )$ on $(A, B)$. Let $G_2 = \left(
\begin{array}{cc}
 p & q \\
 r & s \\
\end{array}
\right) \in \text{GL}_2(\Z )$. It is easy to check that the binary cubic form $4 \det (A x + B y)$ satisfies 
\begin{equation*}
4 \det \left( (p A + q B) x + (r A + s B) y \right) = 4 \det (A x + B y) \circ G_2^{\top} ,
\end{equation*}
where $4 \det (A x + B y) \circ G_2^{\top}$ denotes the replacement 
\begin{equation*}
\left(
\begin{array}{c}
 x \\
 y \\
\end{array}
\right) \longmapsto G_2^{\top} \left(
\begin{array}{c}
 x \\
 y \\
\end{array}
\right) 
\end{equation*}
in $4 \det (A x + B y)$. It is important to note that when we are using binary cubic forms to define cubic fields, these forms will necessarily be irreducible in $\Q [x, y]$ whereas when we discuss binary cubic forms in the context of quartic fields they may have linear factors in $\Q [x, y]$.

We will say that a pair of symmetric half integral matrices $(C,D)$ lies in the same $GL_3(\Z), GL_2(\Z)$-class as $(A,B)$ if and only if there exist matrices $G_3$ and $G_2$ as defined above, such that
\begin{align*}
    \left( \overline{A} , \overline{B} \right) = \left( G_3^{\top} A G_3 , G_3^{\top} B G_3 \right), \quad 
    (C,D) = (p \overline{A} + q \overline{B}, r \overline{A} + s \overline{B}),
\end{align*}
for some symmetric half-integral matrices $\overline{A}$ and $\overline{B}$.

The main goals of this article are to give formulas for two ternary quadratic forms that parametrize a given order of a quartic field and to facilitate the construction of an integral basis of the maximal order of a quartic field using the concept of essential pairs that will be introduced in Section \ref{essentpairs}. Following the work of Bhargava \cite{Bhargavacomp4} and Wood \cite{Wood4}, we will begin in Section \ref{mapforms} by discussing how a pair of ternary quadratic forms produces an order of a quartic field. We illustrate a convenient matrix method for working with these rings, an approach introduced in \cite{hamblarithnf4, cfg4}, and exhibit our technique for constructing a normalized basis for the ring of integers of a quartic field. In Section \ref{getpair} we describe an algorithm for finding a pair of symmetric half-integral matrices that parametrize a given order of a quartic field. To our knowledge, such an algorithm is currently absent from the literature, however O'Dorney \cite{ODorney} has given another technique to find the parameters although our procedure is elementary and simple to use. In order to easily obtain normalized integral bases of number fields, we use the LMFDB database \cite{LMFDB} together with a relatively new method for calculating the integral basis of the ring of integers of a quartic field discussed in Section \ref{essentpairs}. 

\section*{Acknowledgements}

We thank Renate Scheidler for helping us to improve the clarity of our explanations. 

\section{Mapping a pair of ternary quadratic forms to an order}\label{mapforms}

To understand how a pair of ternary quadratic forms gives rise to an quartic ring we will first consider how binary quartic forms produce special kinds of orders of quartic fields, a result due to Wood \cite{Wood4}, however we will place this result in the context of the covariants of binary cubic forms, an observation that is possibly new. The binary quartic form $\mathcal{V} = (a, b, c, d, e)$ has the same discriminant $\Delta $ as the binary cubic form 
\begin{equation}\label{resolvform}
\mathcal{C} = \left( 1, - c, b d - 4 a e, 4 a c e - a d^2 - b^2 e \right) . 
\end{equation}
Denoting the coefficients of $\mathcal{C}$ in \eqref{resolvform} by $A, B, C, D$, we let 
\begin{eqnarray*}
\mathcal{Q} & = & \left( B^2 - 3 A C \right) x^2 + (B C - 9 A D) x y + \left( C^2 - 3 B D \right) y^2 , \\ 
\mathcal{F} & = & \left( -27 A^2 D+9 A B C-2 B^3 \right) x^3 + \left( - 27 A B D + 18 A C^2 - 3 B^2 C \right) x^2 y \\
            &   & + \left( 27 A C D - 18 B^2 D + 3 B C^2 \right) x y^2 + \left( 27 A D^2 - 9 B C D + 2 C^3 \right) y^3 ,
\end{eqnarray*}
the Hessian of $\mathcal{C}$ and the Jacobian $\mathcal{F}$ of $\mathcal{C}$, and let $i = \mathcal{Q}(1, 0)$ and $j = - \mathcal{F}(1, 0)$. Since $\mathcal{C}(1, 0) = 1$, the identity 
\begin{equation}\label{ijinvariants}
j^2 + 27 \Delta = 4 i^3
\end{equation}
is a consequence of Cayley's syzygy $\mathcal{F}^2 + 27 \Delta \mathcal{C}^2 = 4 \mathcal{Q}^3$. Hambleton and Williams \cite[pp. 21--24]{cfg4} wrote much about the role of this syzygy in the study of cubic number fields.  

The binary cubic form $\mathcal{C}$ plays an important role in both solving quartic equations algebraically and in parameterizing quartic rings. We will refer to $\mathcal{C}$ satisfying \eqref{resolvform} as the {\em cubic resolvent form of $\mathcal{V}$}. It is natural to suspect that binary quartic forms parametrize quartic rings, however, subsequent to Bhargava's revelation that pairs of ternary quadratic forms parametrize quartic rings, Wood \cite{Wood4} proved the following result showing that classes of binary quartic forms only parametrize some quartic rings, those generated by a single element. 
\begin{theorem}\label{woodone}
There is a discriminant preserving bijection between the $\text{GL}_2(\Z )$-class of a binary quartic forms and the isomorphism class of the pair $\left( R_4, R_3 \right)$, where $R_4$ is a quartic ring and $R_3$ is a monogenic cubic resolvent ring (generated by a single element). 
\end{theorem}

In Wood's construction, the binary quartic form $\mathcal{V} = (a,b,c,d,e)$ was sent to the pair of symmetric $3 \times 3$ half-integral matrices 
\begin{align*}
A & = \frac{1}{2} \left(
\begin{array}{ccc}
 0 & 1 & 0 \\
 1 & 0 & 0 \\
 0 & 0 & -2 \\
\end{array}
\right) , & B & = \frac{1}{2} \left(
\begin{array}{ccc}
 2 e & 0 & d \\
 0 & 2 a & b \\
 d & b & 2 c \\
\end{array}
\right) ,
\end{align*}
that corresponds to the pair of ternary quadratic forms 
\begin{align*}
\mathcal{Q}_A & = x y - z^2 , & \mathcal{Q}_B & = e x^2 + a y^2 + b y z + d x z + c z^2 .
\end{align*} 
The pair of matrices $(A, B)$ was then mapped to the binary cubic form 
\begin{equation*}
4 \det \left( A x + B y \right) = \mathcal{C}(x, y) , 
\end{equation*}
where $\mathcal{C}$ is given by \eqref{resolvform}. The roots $\zeta_1, \zeta_2, \zeta_3, \zeta_4 $ of $\mathcal{V}(x, 1)$ and the roots $\delta_1, \delta_2, \delta_3 $ of $\mathcal{C}(x, 1)$ are related. We have  
\begin{align}\label{ptrnfdg}
\delta_1 & = a \left( \zeta_1 \zeta_2 + \zeta_3 \zeta_4 \right) , & \delta_2 & = a \left( \zeta_1 \zeta_3 + \zeta_2 \zeta_4 \right) , & \delta_3 & = a \left( \zeta_1 \zeta_4 + \zeta_2 \zeta_3 \right) . 
\end{align}
In summary, the binary quartic form $\mathcal{V} $ is sent to the pair of ternary quadratic forms $\left( \mathcal{Q}_A, \mathcal{Q}_B \right) $ and the associated cubic order is monogenic. To get the quartic order, we map $\left( \mathcal{Q}_A, \mathcal{Q}_B \right)$ to the quartic order by Bhargava's results that we have yet to state in Theorem \ref{Bhthm} below.

\begin{example}\label{extwopointone}
Let $\mathcal{V}(x, y) = (7, 5, -4, 2, 6)$. The cubic resolvent form of $\mathcal{V}$ is the binary cubic form $\mathcal{C} = (1, 4, -158, -850)$ of discriminant $\Delta = 6556372$. Now $\mathcal{C}$ parametrizes an order in a totally real cubic number field by the Levi correspondence. Since the leading coefficient of $\mathcal{C}$ is equal to $1$, this means that the order can be generated by a single element $\delta $. Numerical approximations of the roots of $\mathcal{C}(x, 1)$ are given by 
\begin{align*}
\delta & = 13.0679, & \delta' & = -11.3241, & \delta'' & = -5.7438.
\end{align*}
An integral basis for the order $\mathcal{O}$ parametrized by $\mathcal{C}$ is given by $\left\{ 1, \delta , \delta^2 + 4 \delta \right\}$, or the power basis $\left\{ 1, \delta , \delta^2 \right\}$. The roots of $\mathcal{V}(x, 1)$ are the complex numbers approximated by
\begin{align*}
\zeta & = -0.9794 - 0.3049 i, & \zeta' & = -0.9794 + 0.3049 i, \\
\zeta'' & = 0.6223 - 0.6535 i, & \zeta''' & = 0.6223 + 0.6535 i .
\end{align*}
It is easy to see that \eqref{ptrnfdg} is satisfied since 
\begin{align*}
\delta & = 7 \left( \zeta \zeta' + \zeta'' \zeta''' \right) , & \delta' & = 7 \left( \zeta \zeta'' + \zeta' \zeta''' \right) , & \delta'' & = 7 \left( \zeta \zeta''' + \zeta' \zeta'' \right) .
\end{align*}
\end{example}

Bhargava \cite{Bhargavacomp4} proved that pairs of ternary quadratic forms produce quartic rings in the following result.  
\begin{theorem}\label{Bhthm}
There is a discriminant preserving bijection between the $\text{GL}_3(\Z ) \times \text{GL}_2(\Z )$-class of pairs of ternary quadratic forms and the isomorphism class of the pair $\left( R_4, R_3 \right)$, where $R_4$ is a quartic ring and $R_3$ is a cubic resolvent ring.
\end{theorem}

Theorem \ref{Bhthm} states that classes of ternary quadratic forms parametrize all quartic rings whereas Theorem \ref{woodone} shows that binary quartic forms only parametrize those quartic rings with monogenic cubic resolvent rings. 

To prove Theorem \ref{woodone}, Wood \cite{Wood4} gave a multiplication table for the quartic ring parametrized by the binary quartic form $\mathcal{V}$, so that the map from $\mathcal{V}$ to $R_4$ is explicit. The monogenic cubic resolvent ring is the order parametrized by $- \mathcal{C}$ with $\mathcal{C}$ in \eqref{resolvform}, also easy to provide a multiplication table for. Conversely, if $R_4$ is a quartic ring with a monogenic cubic resolvent ring parametrized by $(1, B, C, D)$, then we can find a binary quartic form $\mathcal{V}$ that parametrizes $R_4$ with cubic resolvent form $(1, B, C, D)$. However again, to be clear, binary quartic forms do not parametrize all quartic rings as the following example illustrates. 
\begin{example}\label{exdfgopa}
Consider the quartic number field $\EE = \Q (\zeta )$ of discriminant $\Delta = 225$, where $\zeta $ is a root of $\mathcal{V}(x, 1)$, where $\mathcal{V} = (4, -6, 5, -3, 1)$ of discriminant $2^2 \Delta $. The ring of integers $\mathcal{O}_{\EE }$ of $\EE$ is generated by the integral basis $$\left\{ 1, \omega_1, \omega_2, \omega_3 \right\} = \left\{ 1, 2 \zeta , 4 \zeta^2 - 6 \zeta , 4 \zeta^3 - 6 \zeta^2 + 5 \zeta \right\} .$$ $\mathcal{O}_{\EE }$ is a quartic ring of discriminant $225$ but there is no binary quartic form of discriminant $225$ because the elliptic curve $y^2 = 4 x^3 - 6075$ has no integral points \cite{MathStackEx} and hence \eqref{ijinvariants} has no solution so the assumption that a binary quartic form of discriminant $225$ is contradicted. We see then, that $\mathcal{O}_{\EE }$ is not parametrized by any binary quartic form. Instead, according to Theorem \ref{Bhthm}, $\mathcal{O}_{\EE }$ is parametrized by a pair of ternary quadratic forms. An arithmetic matrix for $\mathcal{O}_{\EE }$ is given by 
\begin{equation*}
N_{\EE }^{(\alpha )} = \left(
\begin{array}{cccc}
 u & -2 z & 6 z-4 y & -2 x+6 y-5 z \\
 x & u+3 x-5 y+3 z & -5 x+6 y-2 z & 3 x-2 y \\
 y & x & u-5 y+3 z & 3 y-z \\
 z & 2 y & 2 x-6 y & u+3 z \\
\end{array}
\right) ,
\end{equation*}
which provides a convenient means of adding and multiplying the elements $\alpha = u + x \omega_1 + y \omega_2 + z \omega_3 \in \mathcal{O}$ using matrix addition and multiplication, where $u, x, y, z \in \Z$. Providing an arithmetic matrix is equivalent to giving a multiplication table for $\mathcal{O}_{\EE }$. 
\end{example}

Bhargava \cite{Bhargavacomp4} noted that even if we have multiplication tables for the quartic ring and its cubic resolvent ring, it is not obvious how we should obtain the corresponding pair of ternary quadratic forms. Given the pair of $3 \times 3$ symmetric matrices 
\begin{align}\label{pairtern}
A & = \frac{1}{2} \left(
\begin{array}{ccc}
 2 a_{11} & a_{12} & a_{13} \\
 a_{12} & 2 a_{22} & a_{23} \\
 a_{13} & a_{23} & 2 a_{33} \\
\end{array}
\right) , & B & = \frac{1}{2} \left(
\begin{array}{ccc}
 2 b_{11} & b_{12} & b_{13} \\
 b_{12} & 2 b_{22} & b_{23} \\
 b_{13} & b_{23} & 2 b_{33} \\
\end{array}
\right) ,
\end{align}
a multiplication table for the quartic ring $\mathcal{O}$ with normalized basis $\left\{ 1, \omega_1, \omega_2, \omega_3 \right\}$ was written. This was given in terms of coefficients $c_{ij}^{(k)}$ of the basis generators $\omega_k$. 
\begin{equation}\label{flvdfqasd}
\omega_i \omega_j = c_{ij}^{(0)} + c_{ij}^{(1)} \omega_1 + c_{ij}^{(2)} \omega_2 + c_{ij}^{(3)} \omega_3 , 
\end{equation}
where normalized means that $c_{12}^{(1)} = c_{12}^{(2)} = c_{13}^{(1)} = 0$. The identity \eqref{flvdfqasd} provides the following multiplication table.
\begin{eqnarray*}
\omega_1^2 & = & c_{11}^{(0)} + c_{11}^{(1)} \omega_1 + c_{11}^{(2)} \omega_2 + c_{11}^{(3)} \omega_3 , \\
\omega_1 \omega_2 & = & c_{12}^{(0)} + c_{12}^{(1)} \omega_1 + c_{12}^{(2)} \omega_2 + c_{12}^{(3)} \omega_3 , \\
\omega_1 \omega_3 & = & c_{13}^{(0)} + c_{13}^{(1)} \omega_1 + c_{13}^{(2)} \omega_2 + c_{13}^{(3)} \omega_3 , \\
\omega_2^2 & = & c_{22}^{(0)} + c_{22}^{(1)} \omega_1 + c_{22}^{(2)} \omega_2 + c_{22}^{(3)} \omega_3 , \\
\omega_2 \omega_3 & = & c_{23}^{(0)} + c_{23}^{(1)} \omega_1 + c_{23}^{(2)} \omega_2 + c_{23}^{(3)} \omega_3 , \\
\omega_3^2 & = & c_{33}^{(0)} + c_{33}^{(1)} \omega_1 + c_{33}^{(2)} \omega_2 + c_{33}^{(3)} \omega_3 . 
\end{eqnarray*}
From a multiplication table we can construct an arithmetic matrix $N_{\mathcal{O}}^{(\alpha )}$ corresponding to the quartic order $\mathcal{O}$ parametrized by the pair $(A, B)$, an approach introduced in \cite{hamblarithnf4, cfg4}. The matrix $N_{\mathcal{O}}^{(\alpha )}$ is given by 
\small 
\begin{equation}\label{quartpars}
\left(
\begin{array}{cccc}
 u & x c_{11}^{(0)} + y c_{12}^{(0)} + z c_{13}^{(0)} & x c_{12}^{(0)} + y c_{22}^{(0)} + z c_{23}^{(0)} & x c_{13}^{(0)} + y c_{23}^{(0)} + z c_{33}^{(0)} \\
 x & u + x c_{11}^{(1)} + y c_{12}^{(1)} + z c_{13}^{(1)} & x c_{12}^{(1)} + y c_{22}^{(1)} + z c_{23}^{(1)} & x c_{13}^{(1)} + y c_{23}^{(1)} + z c_{33}^{(1)} \\
 y &  x c_{11}^{(2)} + y c_{12}^{(2)} + z c_{13}^{(2)} & u + x c_{12}^{(2)} + y c_{22}^{(2)} + z c_{23}^{(2)} & x c_{13}^{(2)} + y c_{23}^{(2)} + z c_{33}^{(2)} \\
 z &  x c_{11}^{(3)} + y c_{12}^{(3)} + z c_{13}^{(3)} & u + x c_{12}^{(3)} + y c_{22}^{(3)} + z c_{23}^{(3)} & u + x c_{13}^{(3)} + y c_{23}^{(3)} + z c_{33}^{(3)} \\
\end{array}
\right) ,
\end{equation}
\normalsize 
where for $k > 0$ the eighteen $c_{ij}^{(k)}$ are given by 
\begin{align*}
c_{11}^{(1)} & = a_{13} b_{12} - a_{12} b_{13}  & c_{11}^{(2)} & = a_{11} b_{13} - a_{13} b_{11} , & c_{11}^{(3)} & = a_{12} b_{11} - a_{11} b_{12} , \\
             &   \ + a_{23} b_{11} - a_{11} b_{23} , & & & & \\
c_{12}^{(1)} & = 0 , & c_{12}^{(2)} & = 0 , & c_{12}^{(3)} & = a_{22} b_{11} - a_{11} b_{22} , \\
c_{13}^{(1)} & = 0 , & c_{13}^{(2)} & =  a_{11} b_{33} - a_{33} b_{11} , & c_{13}^{(3)} & = a_{23} b_{11} - a_{11} b_{23} , \\
c_{22}^{(1)} & = a_{23} b_{22} - a_{22} b_{23} , & c_{22}^{(2)} & = a_{12} b_{23} - a_{23} b_{12}  & c_{22}^{(3)} & = a_{22} b_{12} - a_{12} b_{22} , \\
             &                                   &              & \ + a_{22} b_{13} - a_{13} b_{22}  , & & \\ 
c_{23}^{(1)} & = a_{33} b_{22} - a_{22} b_{33} , & c_{23}^{(2)} & = a_{12} b_{33} - a_{33} b_{12} , & c_{23}^{(3)} & = a_{22} b_{13} - a_{13} b_{22} , \\
c_{33}^{(1)} & = a_{33} b_{23} - a_{23} b_{33} , & c_{33}^{(2)} & = a_{13} b_{33} - a_{33} b_{13} , & c_{33}^{(3)} & = a_{12} b_{33} - a_{33} b_{12} \\
             &                                   &              &                                 &              & \ + a_{23} b_{13} - a_{13} b_{23} .
\end{align*}
Let 
\begin{align*}
\alpha_1 & = u_1 + x_1 \omega_1 + y_1 \omega_2 + z_1 \omega_3 , & \alpha_2 & = u_2 + x_2 \omega_1 + y_2 \omega_2 + z_2 \omega_3 . 
\end{align*}
Next we solve the equation 
\begin{equation*}
N_{\mathcal{O}}^{\left( \alpha_1 \right) } N_{\mathcal{O}}^{\left( \alpha_2 \right) } - N_{\mathcal{O}}^{\left( \alpha_2 \right) } N_{\mathcal{O}}^{\left( \alpha_1 \right) } = [0]
\end{equation*}
for $c_{11}^{(0)}$, $c_{12}^{(0)}$, $c_{13}^{(0)}$, $c_{22}^{(0)}$, $c_{23}^{(0)}$, and $c_{33}^{(0)}$ and find that we must have 
\begin{eqnarray*}
c_{11}^{(0)} & = & \left( c_{12}^{(2)} - c_{11}^{(1)} \right) c_{12}^{(2)} + c_{12}^{(3)} c_{13}^{(2)} + c_{11}^{(2)} \left( c_{12}^{(1)} - c_{22}^{(2)} \right) - c_{11}^{(3)} c_{23}^{(2)} , \\
c_{12}^{(0)} & = & c_{11}^{(2)} c_{22}^{(1)} + c_{11}^{(3)} c_{23}^{(1)} - c_{12}^{(1)} c_{12}^{(2)} - c_{12}^{(3)} c_{13}^{(1)} , \\
c_{13}^{(0)} & = & c_{11}^{(2)} c_{23}^{(1)} + c_{11}^{(3)} c_{33}^{(1)} - c_{12}^{(1)} c_{13}^{(2)} - c_{13}^{(1)} c_{13}^{(3)} , \\
c_{22}^{(0)} & = & \left( c_{12}^{(1)} + c_{12}^{(2)} - c_{11}^{(1)} \right) c_{22}^{(1)}  + c_{12}^{(3)} c_{23}^{(1)} - c_{12}^{(1)} c_{22}^{(2)} - c_{13}^{(1)} c_{22}^{(3)} , \\
c_{23}^{(0)} & = &  c_{12}^{(1)} \left( c_{13}^{(1)} - c_{23}^{(2)} \right) - c_{11}^{(1)} c_{23}^{(1)} + c_{12}^{(2)} c_{23}^{(1)} - c_{13}^{(1)} c_{23}^{(3)} + c_{12}^{(3)} c_{33}^{(1)} , \\
c_{33}^{(0)} & = & \left( c_{13}^{(3)} - c_{11}^{(1)} \right) c_{33}^{(1)} + c_{13}^{(1)} \left( c_{13}^{
(1)} - c_{33}^{(3)} \right) + c_{13}^{(2)} c_{23}^{(1)} - c_{12}^{(1)} c_{33}^{(2)} .
\end{eqnarray*}
Since the arithmetic matrix $N_{\mathcal{O}}^{(\alpha )}$ indeed commutes, we see that the pair $(A, B)$ given by \eqref{pairtern} produces a ring with arithmetic matrix \eqref{quartpars}. 

If the $a_{ij}$ and $b_{ij}$ are rational integers and at least one of the minimum polynomials $f_1(x)$, $f_2(x)$, and $f_3(x)$ of $\omega_1$, $\omega_2$, and $\omega_3$ given by the identity
\begin{equation}\label{minpolydfg}
f_j(x) = \det \left( N_{\mathcal{O}}^{ \left( x - \omega_j \right) } \right) 
\end{equation}
has  degree equal to $4$, then $\left\{ 1, \omega_1 , \omega_2, \omega_3 \right\}$ is an integral basis for an order $\mathcal{O}$ of a quartic number field $\EE$. These polynomials are monic and hence $\omega_1 , \omega_2, \omega_3$ are algebraic integers. It can be shown that the polynomials $f_2(x)$ and $f_3(x)$ have the same discriminant, and the quotient of the discriminants of $f_1(x)$ and $f_2(x)$ is the square of a rational number.

\begin{example}\label{trfofs}
Theorem \ref{Bhthm} permits us to choose any twelve rational integers and obtain an arithmetic matrix for a ring. Of course some choices will not lead to genuine orders of quartic fields. Taking integers from the collection\\ $\{ 2,-5,3,3,1,3,0,-4,3,-3,1,-3 \}$, let 
\begin{align*}
A & = \frac{1}{2} \left(
\begin{array}{ccc}
 4 & -5 & 3 \\
 -5 & 6 & 1 \\
 3 & 1 & 6 \\
\end{array}
\right) , & B & = \frac{1}{2} \left(
\begin{array}{ccc}
 0 & -4 & 3 \\
 -4 & -6 & 1 \\
 3 & 1 & -6 \\
\end{array}
\right) . 
\end{align*} 
Computing the $24$ values of $c_{ij}^{(k)}$ for $1 \leq i < j \leq 3$ and $0 \leq k \leq 3$, we obtain the arithmetic matrix 
\begin{equation*}
N_{\mathcal{O}}^{(\alpha )} = \left(
\begin{array}{cccc}
 u & -354 x-36 y+48 z & -36 x+6 y+36 z & 48 x+36 y-18 z \\
 x & u+x & -6 y & 6 z \\
 y & 6 x-6 z & u+17 y+27 z & -6 x+27 y-18 z \\
 z & 8 x+6 y-2 z & 6 x-27 y+18 z & u-2 x+18 y+27 z \\
\end{array}
\right) ,  
\end{equation*}
which gives the ring structure of $\mathcal{O}$ parametrized by the pair of ternary quadratic forms 
\begin{align*}
\mathcal{Q}_1 & = 2 x^2 - 5 x y + 3 x z + 3 y^2 + y z + 3 z^2 , & \mathcal{Q}_2 & =  - 4 x y + 3 x z - 3 y^2 + y z - 3 z^2 . 
\end{align*}
The binary cubic form 
\begin{equation*}
\mathcal{C}(x, y) = 4 \det (A x + B y) = - 47 x^3 -262 x^2 y + 130 x y^2 + 63 y^3 
\end{equation*}
has discriminant $7683877869$. Approximations of the roots of $\mathcal{C}(x, 1)$ are given by
\begin{align*}
\delta_1 & = -1.5594, & \delta_2 & = -0.3039, & \delta_3 & = 4.8878.
\end{align*}
We denote the generators of the integral basis of $\mathcal{O}$ by the elements of $\left\{ 1, \omega_1, \omega_2, \omega_3 \right\}$. Minimum polynomials $f_j(x)$ of the $\omega_j$ can be found using the identity \eqref{minpolydfg}. We have 
\begin{eqnarray*}
f_1(x) & = & x^4 + x^3 + 388 x^2 + 504 x + 9720 , \\
f_2(x) & = & x^4 - 35 x^3 + 1029 x^2 + 1836 x + 5184 , \\
f_3(x) & = & x^4 - 54 x^3 + 1083 x^2 + 198 x + 9072 .
\end{eqnarray*}
Next we give numerical approximations of the roots $\omega_j^{(k)}$ of the $f_j(x)$ sorted so that their products agree with multiplication in $\mathcal{O}$ determined by multiplying the matrices $N_{\mathcal{O}}^{(\alpha )}$. We put these roots in the following matrix 
\begin{eqnarray}\label{gammamatr}
\Gamma_{\mathcal{O}} & = & \left(
\begin{array}{cccc}
 1 & \omega_1 & \omega_2 & \omega_3 \\
 1 & \omega_1^{(1)} & \omega_2^{(1)} & \omega_3^{(1)} \\
 1 & \omega_1^{(2)} & \omega_2^{(2)} & \omega_3^{(2)} \\
 1 & \omega_1^{(3)} & \omega_2^{(3)} & \omega_3^{(3)} \\
\end{array}
\right) , \\
\nonumber & \approx & \left(
\begin{array}{cccc}
 1 & -0.71-5.13 i & 18.4+27.4 i & 27.3-19.0 i \\
 1 & -0.71+5.13 i & 18.4-27.4 i & 27.3+19.0 i \\
 1 & 0.2-19.0 i & -0.92-1.97 i & -0.29+2.85 i \\
 1 & 0.2+19.0 i & -0.92+1.97 i & -0.29-2.85 i \\
\end{array}
\right) .
\end{eqnarray} 
Let 
\begin{align*}
\alpha & = u + x \omega_{1} + y \omega _{2} + z \omega_{3} , & \alpha' & = u + x \omega_{1}^{(1)} + y \omega_{2}^{(1)} + z \omega_{3}^{(1)} , \\
\alpha'' & = u + x \omega_{1}^{(2)} + y \omega_{2}^{(2)} + z \omega_{3}^{(2)} , & \alpha''' & = u + x \omega_{1}^{(3)} + y \omega_{2}^{(3)} + z \omega_{3}^{(3)} ,
\end{align*}
and 
\begin{equation*}
\Theta^{(\alpha )} = \left(
\begin{array}{cccc}
 \alpha & 0 & 0 & 0 \\
 0 & \alpha' & 0 & 0 \\
 0 & 0 & \alpha'' & 0 \\
 0 & 0 & 0 & \alpha''' \\
\end{array}
\right) .
\end{equation*}
Then we have 
\begin{equation}\label{diagform}
\left( \Gamma_{\mathcal{O}} \right)^{-1} \Theta^{(\alpha )} \Gamma_{\mathcal{O}} = N_{\mathcal{O}}^{(\alpha )}. 
\end{equation}
Notice that $\Delta = \det \left( \Gamma_{\mathcal{O}} \right)^2 = 7683877869$, which is square-free, and that this is equal to the discriminant of $\mathcal{C}(x, y)$. The polynomials $f_1(x)$, $f_2(x)$, and $f_3(x)$ have discriminants $\Delta \cdot 504^2$, $\Delta \cdot 3132^2$, and $\Delta \cdot 3132^2$. In the case of this example the pair $(A, B)$ parametrizes the ring of integers of a quartic field of discriminant $7683877869$. 
\end{example}

To understand that all bases of orders of quartic fields can be normalized, we will consider how a change of integral basis of an order of a quartic field affects the arithmetic matrix $N_{\mathcal{O}}^{(\alpha )}$. For this discussion we assume that we have two ordered integral bases of an order $\mathcal{O}$ given by $\mathcal{R} = \left\{ 1, \omega_1, \omega_2, \omega_3 \right\}$ and $\mathcal{S} = \left\{ 1, \rho_1, \rho_2, \rho_3 \right\}$, furthermore we call the matrix that we usually refer to as $\Gamma_{\mathcal{O}}$ instead $\Gamma_{\mathcal{R}}$ and $\Gamma_{\mathcal{S}}$ and rather than writing $N_{\mathcal{O}}^{(\alpha )}$ we will put $N_{\mathcal{R}}^{(\alpha )}$ and $N_{\mathcal{S}}^{(\alpha )}$. A change of basis is given by 
\begin{equation*}
\Gamma_{\mathcal{R}} = \Gamma_{\mathcal{S}} G,
\end{equation*}
where $G \in \text{GL}_4(\Z )$. The diagonalized expression for $N_{\mathcal{R}}^{(\alpha )}$ and $N_{\mathcal{S}}^{(\alpha )}$ can be used to show that 
\begin{equation}\label{frosdf}
N_{\mathcal{R}}^{(\alpha )} = G^{-1} N_{\mathcal{S}}^{(\alpha )} G .
\end{equation}
The pair of symmetric matrices $(A, B)$ gives rise to a normalized basis, however a change of basis by multiplication by $G$, given by \eqref{frosdf}, will produce an arithmetic matrix that does not necessarily correspond to a normalized basis. 

\section{Essential pairs}\label{essentpairs}

Under certain circumstances we can give a formula for an arithmetic matrix corresponding to a normalized basis of the maximal order of a quartic field of discriminant $\Delta $, and hence formulas the $c_{ij}^{(k)}$. We will call the pair $\left[ f , \ (a,b,c,d,e) \right]$ an {\em essential pair} if the discriminant of the irreducible binary quartic form $\mathcal{V}(x, y) = (a,b,c,d,e)$ is equal to $\Delta f^2$ for some rational integer $f$ such that $f^2$ divides $a$ and $f$ divides $b$, where $\Delta $ is the discriminant of a quartic field $\Q (\zeta )$ with $\mathcal{V}(\zeta , 1) = 0$. In \cite{hamblarithnf4} it was shown that if we have an essential pair, then it is easy to write a basis for the maximal order. We reproduce this result below, omitting details of the proof. 

\begin{theorem}\label{newbasthm}
Let 
\begin{equation}\label{eqfpolt}
p(x) = a x^4 + b x^3 + c x^2 + d x + e 
\end{equation}
be an irreducible polynomial of degree $4$ in $\Z [x]$ such that $\left[ f , \ \mathcal{V}(x, y) \right]$ is an essential pair, where $p(x) = \mathcal{V}(x, 1)$, the discriminant of $\mathcal{V}$ is equal to $\Delta f^2$, $\Delta $ is the discriminant of a number field $\EE = \Q (\zeta )$ of degree $4$ over $\Q $, where $p(\zeta ) = 0$. Then an integral basis for the ring of integers of $\EE$ is given by $\left\{ \omega_0 , \omega_1 , \omega_2 , \omega_3 \right\}$, where 
\begin{align}\label{ibsn}
 \omega_0 & = 1, & \omega_1 & = \frac{a}{f} \zeta , & \omega_2 & = a \zeta^2 + b \zeta , & \omega_3 & = a \zeta^3 + b \zeta^2 + c \zeta .
\end{align}
\end{theorem}

We now sketch some main ideas of the proof. Let $\zeta_i $ be the roots of $p(x)$, where $\zeta_0 = \zeta $, 
\begin{align}\label{wolk}
\Xi & = \left(
\begin{array}{cccc}
 1 & \zeta &  \zeta^2 & \zeta^{3} \\
 1 & \zeta_{1} & \zeta_{1}^2 & \zeta_{1}^{3} \\
 1 & \zeta_{2} & \zeta_{2}^2 & \zeta_{2}^{3} \\
 1 & \zeta_{3} & \zeta_{3}^2 & \zeta_{3}^{3} \\
\end{array}
\right) , & \overline{A} & = \left(
\begin{array}{cccc}
 1 & 0 & 0 & 0  \\
 0 & a & b & c  \\
 0 & 0 & a & b  \\
 0 & 0 & 0 & a  \\
\end{array}
\right) \left(
\begin{array}{cccc}
 1 & 0 & 0 & 0  \\
 0 & f^{-1} & 0 & 0 \\
 0 & 0 & 1 & 0  \\
 0 & 0 & 0 & 1  \\
\end{array}
\right)  ,
\end{align}
and let $\Gamma_{\EE } = \left[ \kappa_{i-1} \left( \omega_{j-1} \right) \right]$, where $\kappa_{i-1}$ for $i = 1$ to $4$ are the embeddings of $\EE $ in $\C$, then we have $\Gamma_{\EE } = \Xi \overline{A} $. The entries of $\overline{A}$ are rational integers. Taking the square of the determinants of these matrices, 
\begin{equation}\label{discnew}
 \det \left( \Gamma_{\EE} \right)^2 = \frac{a^6}{f^2} \prod_{i < j} \left( \zeta_{i} - \zeta_{j} \right)^2 = \Delta . 
\end{equation}
To show that $\omega_1 $ is an algebraic integer, multiplying \eqref{eqfpolt} by the rational integer $\frac{a^{3}}{f^4}$, we see that $\omega_1$ is a root of the monic polynomial 
\begin{equation*}
X^4 + \frac{b}{f} X^{3} + \frac{a c}{f^2} X^{2} + \frac{a^2 d}{f^3} X + \frac{a^{3} e}{f^{4}} ,
\end{equation*}
with coefficients in $\Z $. Finally, it must be shown that the remaining $\omega_j$ are algebraic integers. We leave the remaining details to \cite{hamblarithnf4}. 

A simple translation of the formula for this basis provides a normalized basis with the following arithmetic matrix: 
\begin{equation}\label{arithnormal}
N_{\mathcal{O}}^{(\alpha )} = \left(
\begin{array}{cccc}
 u & -\frac{a c}{f^2} x - \frac{a d}{f} y -\frac{a e}{f} z & - \frac{a d}{f} x - b d y - a e y - b e z & - \frac{a e}{f} x - b e y \\
 x & u - \frac{b}{f} x & - d f y - e f z & - e f y \\
 y & \frac{a}{f^2} x & u + c y & - e z \\
 z & \frac{a}{f} y & \frac{a}{f} x + b y + c z & u + c y + d z \\
\end{array}
\right) .
\end{equation}
When we compare \eqref{arithnormal} with \eqref{quartpars}, we can easily obtain formulas for the $c_{ij}^{(k)}$. If an essential pair exists, then $(A, B)$ parametrizes $\mathcal{O}$, where 
\begin{align*}
A & = \frac{1}{2} \left(
\begin{array}{ccc}
 2 \frac{a}{f^2} & \frac{b}{f} & 0 \\
 \frac{b}{f} & 2 c & d \\
 0 & d & 2 e \\
\end{array}
\right) , & B & = \frac{1}{2} \left(
\begin{array}{ccc}
 0 & 0 & 1 \\
 0 & -2 f & 0 \\
 1 & 0 & 0 \\
\end{array}
\right) , 
\end{align*}
\begin{equation*}
4 \det (A x + B y) = - \frac{1}{f^2}\left( a d^2 + b^2 e - 4 a c e \right) x^3 + \frac{1}{f} (b d - 4 a e) x^2 y - c x y^2 + f y^3 . 
\end{equation*}
This is part of the method that we will later use to tabulate pairs of ternary quadratic forms that parametrize maximal orders of quartic fields. When they exist, and are surprisingly common, essential pairs are generally easy to find as the following result shows. It characterizes the quartic fields that have an essential pair, provides a test for whether they exist, and gives a simple method for obtaining them. 
\begin{theorem}\label{resolve}
Let $\EE$ be a quartic field of discriminant $\Delta $ with maximal order $\mathcal{O}$ parametrized by $(A, B)$ and let $S$ denote the cubic resolvent ring of $\mathcal{O}$ parametrized by the binary cubic form $4 \det (A x + B y)$, with integral basis $\left\{ 1, \rho_1, \rho_2 \right\}$. Statements (1), (2) and (3) are equivalent and (3) implies (4). 
\begin{enumerate}
\item There is an essential pair $[f, (a, b, c, d, e)]$ for the quartic field $\EE$.
\item There is a monic quartic polynomial $g(x) \in \Z [x]$ of discriminant $\Delta f^2$ with a root generating $\EE$ and an integer $t$ such that 
\small 
\begin{align}\label{thethree}
g(t) & \equiv 0 \pmod{f^2}, & g'(t) & \equiv 0 \pmod{f}, 
\end{align}
\normalsize 
where $g'(x)$ denotes the first derivative; $g(x)$ has a root in $\Z /f^2$ that is repeated in $\Z / f$. 
\item There is an element $\tau = \ell + m \rho_1 + n \rho_2$ of $S$ such that $S/ \Z [\tau ] $ is cyclic of order $f$ and $\mathcal{Q}_{m A + n B}(x, y, z)$ belongs to the same $\text{GL}_3(\Z)$-class as $x z - f y^2$. 
\item There is an ideal $\mathfrak{a} $ of $S$ such that $S/ \mathfrak{a}$ is isomorphic to $\Z / f$. 
\end{enumerate}
\end{theorem}

\begin{proof}
We prove the logical relationships between the items of the theorem in the following steps, where $\Longrightarrow $ denotes implies: $(2) \Longrightarrow (1)$, $(1) \Longrightarrow (2)$, $(1) \Longrightarrow (3)$, $(1) \Longrightarrow (4)$, $(3) \Longrightarrow (1)$.  

If there exists a monic quartic polynomial $g(x)$ with a root generating $\EE$ and an integer $t$ satisfying \eqref{thethree}, then constructing the binary quartic form $\mathcal{V}(x, y) = y^4 g(x/y)$, we have 
\begin{align*}
\mathcal{V}(t, 1) & \equiv 0 \pmod{f^2}, & \mathcal{V}_x(t, 1) & \equiv 0 \pmod{f}, & \mathcal{V}_y(t, 1) & \equiv 0 \pmod{f} .  
\end{align*} 
There exist $q, s \in \Z$ satisfying $s t - q = \pm 1$. Letting $M = \left(
\begin{array}{cc}
 t & q \\
 1 & s \\
\end{array}
\right) \ \in \text{GL}_2(\Z )$, and $(a,b,c,d,e) = \mathcal{V} \circ M$, it follows that $[f, (a,b,c,d,e) ]$ is an essential pair for $\EE$. We can even choose $q = 1$ and $s = 0$ to obtain the essential pair $$\left[ f, \left( g(t), g'(t), \frac{1}{2} g^{(2)}(t), \frac{1}{6} g^{(3)}(t), 1 \right) \right] .$$ Conversely, if $[f, (a,b,c,d,e) ]$ is an essential pair for $\EE$, then letting $p, q, r, s \in \Z$ satisfy $p s - q r = \pm 1$ with $\gcd(r, f) = 1$, $M = \left(
\begin{array}{cc}
 p & q \\
 r & s \\
\end{array}
\right)$, $\mathcal{V} = (a,b,c,d,e) \circ M^{-1}$, and we let $h(x) = \mathcal{V}(x, 1)$, we have 
\begin{align*}
a & = \mathcal{V}(p, r) \equiv 0 \pmod{f^2}, & b & = q \mathcal{V}_x(p, r) + s \mathcal{V}_y(p, r) \equiv 0 \pmod{f}, 
\end{align*}
and since 
\begin{equation*}
\left(
\begin{array}{cc}
 4 a & b \\
\end{array}
\right) = \left(
\begin{array}{cc}
 \mathcal{V}_x(p, r) & \mathcal{V}_y(p, r) \\
\end{array}
\right) M , 
\end{equation*}
we have 
\begin{align*}
\mathcal{V}(p, r) & \equiv 0 \pmod{f^2} , & \mathcal{V}_x(p, r) & \equiv 0 \pmod{f} .  
\end{align*}
Since $\gcd (r, f) = 1$, there exists $u \in \Z$ satisfying $h(u) \equiv 0 \pmod{f^2}$, $h'(u) \equiv 0 \pmod{f}$. Now letting $X = A x$ and $g(X) = A^3 h(x)$, where $A$ is the leading coefficient of $h(x)$, we obtain a monic polynomial $g(X)$ with a root generating $\EE$ since 
\begin{equation*}
 a x^4 + b x^3 + c x^2 + d x + e = (r x + s)^4 \mathcal{V} \left( \frac{p x + q}{r x + s} , 1 \right) .
\end{equation*}
Furthermore, letting $t = A u$, we see that $t$ satisfies \eqref{thethree}. We can assume with no loss of generality that there is an integer $t$ and an essential pair $[f, (a,b,c,d,e)]$ such that 
\begin{equation*}
(a,b,c,d,e) = \sum_{n = 0}^{4} \frac{1}{n ! } g^{(n)}(t) \ x^{4 - n} y^{n - 4} . 
\end{equation*}
By Taylor's theorem we must have 
\begin{equation*}
g( x ) = a + (x-t) \left( b + c (x-t) + d (x-t)^2 + e (x-t)^3 \right) . 
\end{equation*}
We see that $\EE$ has an essential pair if and only if there is a monic polynomial $g(x) \in \Z [x]$ that has a root generating $\EE$ and $g(x)$ has a root in $\Z /f^2$ that is repeated in $\Z / f$. 

Let $[f, (a,b,c,d,e)]$ be an essential pair for the quartic field $\EE$ with maximal order $\mathcal{O}$ and let 
\begin{align}\label{pandq}
p & = - \frac{1}{f^2}\left( a d^2 + b^2 e - 4 a c e \right) , & q & = \frac{1}{f} (b d - 4 a e) .  
\end{align} 
We know that $\mathcal{C}(x, y) = (p, q, -c, f)$ is an index form for the cubic resolvent ring $S$ of $\mathcal{O}$. Let $\delta $ be a root of $\mathcal{C}(x, 1)$, $\rho_1 = p \delta $, $\rho_2 = p \delta^2  + q \delta $. Then $\left\{ 1, \rho_1, \rho_2 \right\}$ is an integral basis for $S$ and we have arithmetic matrix 
\begin{equation*}
N_S^{(\alpha )} = \left(
\begin{array}{ccc}
 u & - f p y & - f p x - f q y \\
 x & u - q x + c y & c x - f y \\
 y & p x & u + c y \\
\end{array}
\right) ,
\end{equation*}
where $\alpha_j = u_j + x_j \rho_1 + y_j \rho_2$. Following a technique presented in \cite[pp. 19]{cfg4} using a different basis, we have 
\begin{equation*}
\left(
\begin{array}{ccc}
 1 & \alpha  & \alpha^2 \\
 1 & \alpha'  & \alpha'^2 \\
 1 & \alpha''  & \alpha''^2 \\
\end{array}
\right) = 
\left(
\begin{array}{ccc}
 1 & \rho_1  & \rho_2 \\
 1 & \rho_1'  & \rho_2' \\
 1 & \rho_1''  & \rho_2'' \\
\end{array}
\right) \left(
\begin{array}{ccc}
 1 & u & U \\
 0 & x & X \\
 0 & y & Y \\
\end{array}
\right) , 
\end{equation*}
where $\alpha^2 = U + X \rho_1 + Y \rho_2$ and 
\begin{align*}
U & = u^2 - 2 p f x y - q f y^2 , & X & = 2 u x - q x^2 + 2 c x y - f y^2, & Y & = 2 u y + p x^2 + c y^2 .  
\end{align*}
Taking determinants and squaring the result, the discriminant of $\alpha $ is equal to $(x Y - y X)^2 \Delta$ and since $x Y - y X = \mathcal{C}(x, y)$, the index of $\alpha $ is equal to $\left| \mathcal{C}(x, y) \right|$. Now $\mathcal{C}(0, 1) = f$ so for all $u \in \Z$ the element $\tau = u + \rho_2 \in S$ has index $f$. In particular, since $f \rho_1 = - f q + c \rho_2 - \rho_2^2$, we see that the quotient module $S/ \Z \left[ \rho_2 \right]$ is cyclic of order $f$ with element $n \rho_1 + \Z \left[ \rho_2 \right]$, $n = 0, 1, \dots f - 1$. 

To show that (1) implies (4), notice that the product of $\alpha_1, \alpha_2 \in S$ is of the form 
\begin{equation*}
u_1 u_2 - f \left( p x_2 y_1 + p x_1 y_2 + q y_2 y_1 \right) + x_3 \rho_1 + y_3 \rho_2  
\end{equation*}
for some $x_3, y_3 \in \Z$ obtained from the left column of $N_S^{(\alpha_1 )} N_S^{(\alpha_2 )}$. We see that there is a surjective ring homomorphism 
\begin{align*}
\psi & : S \longrightarrow \Z / f , & u + x \rho_1 + y \rho_2 & \longmapsto u \pmod{f} . 
\end{align*}
The kernel of the map $\psi $ is the ideal $\mathfrak{a}$ of $S$ given by
\begin{equation*}
\mathfrak{a} = \left\{ u + x \rho_1 + y \rho_2 \in S \ : \ u, x, y \in \Z, \ u \equiv 0 \pmod{f} \right\} .
\end{equation*}
Hence we have the exact sequence of ring homomorphisms
\begin{equation*}
1 \longrightarrow \mathfrak{a} \stackrel{\phi }{\longrightarrow } S \stackrel{\psi }{\longrightarrow } \Z / f \longrightarrow 1 ,
\end{equation*}
where $\phi $ is the inclusion map. It follows that $S/ \mathfrak{a} $ is isomorphic to $\Z / f$. 

Conversely, to prove (3) implies (1), let $g(x) \in \Z [x]$ be a monic polynomial with a root generating $\EE$, of polynomial discriminant $\Delta f^2$, where $\Delta $ is the discriminant of $\EE$. Let $S$ be the cubic resolvent of the maximal order $\mathcal{O}$ of $\EE$ parametrized by $(A, B)$ and let $\mathcal{C}(x, y) = 4 \det (A x + B y)$ be a binary cubic form parametrizing $S$. Since there is an element $\tau $ of $S$ such that $S / \Z[\tau]$ is cyclic of order $f$, the index of $\tau $ is equal to $f$ and hence there exist $m, n \in \Z$ such that $\mathcal{C}(m, n) = f$. This means that we can find $M = \left(
\begin{array}{cc}
 v & m \\
 w & n \\
\end{array}
\right) \in \text{GL}_2(\Z )$ such that $\mathcal{C} \circ M = (p, q, -r, f) = 4 \det \left( \overline{A} x + \overline{B} y \right)$, for some $p, q, r \in \Z$, where $\overline{A} = v A + w B$, $\overline{B} = m A + n B$. Let 
\begin{align*}
\mathcal{Q}_{\overline{A}} & = \left( a_{11}, a_{12}, a_{13}, a_{22}, a_{23}, a_{33} \right) , & \mathcal{Q}_{\overline{B}} & = \left( b_{11}, b_{12}, b_{13}, b_{22}, b_{23}, b_{33} \right) , & G & = \left[ u_{ij} \right] . 
\end{align*} 
By the assumption of Item (3) we can solve the matrix equation 
\begin{equation*}
G^{\top} \left(
\begin{array}{ccc}
 0 & 0 & 1 \\
 0 & - 2 f & 0 \\
 1 & 0 & 0 \\
\end{array}
\right) G = 2 \overline{B} 
\end{equation*}
for $G \in \text{GL}_2(\Z )$. Let $H = G^{-1}$, and $A' = H^{\top} \overline{A} H $, $B' = H^{\top} \overline{B} H $.
Without loss of generality we can assume that $a_{13}' = 0$ since we can transform the binary cubic form $4 \det \left( A' x + B' y \right)$ by $\left(
\begin{array}{cc}
 1 & 0 \\
 k & 1 \\
\end{array}
\right)$, where $k \in \Z $ achieves $a_{13}' = 0$ without changing $B'$. We have 
\begin{equation*}
4 \det \left( A' x + B' y \right) = 4 \det \left( A' \right) x^3 + s x^2 y - r x y^2 + f y^3 .
\end{equation*}
where 
\begin{align*}
 s & = a_{13}'^2 f - 2 a_{22}' a_{13}' + a_{12}' a_{23}' - 4 a_{11}' a_{33}' f , & r & = a_{22}' - 2 a_{13}' f .
\end{align*}
Since $a_{13}' = 0$, we let 
\begin{equation*}
\left( a, b, c, d, e \right) = \left( a_{11}' f^2, a_{12}' f, a_{22}', a_{23}', a_{33}' \right) 
\end{equation*}
and find that $[f, (a, b, c, d, e)]$ is an essential pair for $\EE$. If $\zeta_1, \zeta_2, \zeta_3, \zeta_4$ are the roots of $a x^4 + b x^3 + c x^2 + d x + e$, where $[f, (a,b,c,d,e)]$ is an essential pair, then the roots of $f x^3 - c x^2 + q x + p$ are 
\begin{align*}
\theta_1 & = \frac{a}{f} \left( \zeta_1 \zeta_2 + \zeta_3 \zeta_4 \right) , & \theta_2 & = \frac{a}{f} \left( \zeta_1 \zeta_3 + \zeta_2 \zeta_4 \right) ,  & \theta_3 & = \frac{a}{f} \left( \zeta_1 \zeta_4 + \zeta_2 \zeta_3 \right) , 
\end{align*}
where $p, q$ are of the form given in \eqref{pandq}. 
\end{proof}

This result can make it easier to find a pair of ternary quadratic forms that parametrize the ring of integers of a quartic field using a number field database such as \cite{LMFDB}. For a given polynomial this requires $O \left( f^2 \right)$ operations to test. While solutions to the congruence \eqref{thethree} frequently exist, they do not always exist. If for a particular quartic field $\EE$ \eqref{thethree} has no solution, we can still calculate an integral basis for the ring of integers $\mathcal{O}_{\EE }$, normalize that basis, construct the matrix $\Gamma_{\EE }$, and then calculate an arithmetic matrix for $\mathcal{O}_{\EE }$ using \eqref{diagform} or other means. We note that all maximal orders of quartic fields of positive discriminant less than $576$ have an essential pair. Furthermore, of the quartic polynomials given in the LMFDB database \cite{LMFDB} generating fields of positive and respectively negative discriminant of absolute value less than $10^5$, $78.89 \%$ and $80.22 \%$ of those have a solution $t$ to \eqref{thethree}. 

The following example illustrates calculations in Theorem \ref{resolve}. 
\begin{example}
Consider the essential pair $[3, (9, -3, -5, 1, 1)]$ for the quartic field $\EE $ of discriminant $1197$. Let $M = \left(
\begin{array}{cc}
 3 & 2 \\
 4 & 3 \\
\end{array}
\right)$. We have $$\mathcal{V}(x, y) = (9, -3, -5, 1, 1) \circ M^{-1} = 397 x^4 - 1003 x^3 y + 949 x^2 y^2 - 399 x y^3 + 63 y^4 $$ and we let $h(x) = 397 x^4 - 1003 x^3 + 949 x^2 - 399 x + 63$. Since $4^{-1} \equiv 7 \pmod{f^2}$, and $3 \cdot 7 \equiv 3 \pmod{f^2}$, we let $u = 3$ and obtain $h(u) \equiv 0 \pmod{f^2}$ and $h'(u) \equiv 0 \pmod{f}$. The leading coefficient of $h(x)$ is $A = 397$ and $g(X) = 397^3 h(x) = X^4 - 1003 X^3 + 376753 X^2 - 62885991 X + 3941958699$, where $X = 397 x$. We let $t = 397 u = 1191$ and we have $g(t) \equiv 0 \pmod{f^2}$ and $g'(t) \equiv 0 \pmod{f}$. The factorizations of $g(X)$ in $\Z/ f^2$ and in $\Z/f$ are $X (X + 6) (X^2 + 8 X + 1)$ and $X^2 (X + 1)^2$.  
\end{example}

\begin{example}
In Example \ref{secondeg} we will find that the quartic field $\EE$ of discriminant $1424$ generated by a root of $g(x) = x^4 - 2 x^3 - x^2 + 2 x + 5$ has maximal order $\mathcal{O}$ parametrized by 
\begin{align*}
\mathcal{Q}_A & = (1, 0, 0, 1, -1, -1), & \mathcal{Q}_B & = (0, 0, 2, -1, 1, 3) . 
\end{align*}
with binary cubic form $\mathcal{C}(x, y) = 4 \det (A x + B y) = (-5, 18, -17, 4)$ of discriminant $1424$. The discriminant of $g(x)$ is equal to $1424 \cdot 4^2$. Let $\delta \approx 0.3566$ be a root of $\mathcal{C}(x, 1)$. The cubic resolvent ring $S$ of $\mathcal{O}$ has an element $\rho_2 = - 5 \delta^2 + 18 \delta $ of index $4$. However, the ternary quadratic forms $\mathcal{Q}_B = 2 x z - y^2 + y z + 3 z^2$ and $x z - 4 y^2$ belong to distinct $\text{GL}_3(\Z )$ classes. There are no solutions to the simultaneous congruence $g(t) \equiv 0 \mod{16}$, $g'(t) \equiv 0 \mod{4}$. 
\end{example}

\begin{example}
Let $\EE$ be the field of discriminant $1161$ generated by a root of $g(x) = x^4 - x^3 + 6 x^2 - x + 7$ with index $f = 4$. The simultaneous congruence $g(t) \equiv 0 \pmod{16}$, $g'(t) \equiv 4$ has solutions $t = 3, 7, 11, 15$ so $\EE$ has an essential pair. The maximal order $\mathcal{O}$ of $\EE$ is parametrized by $(A, B)$, where
\begin{align*}
A & = \frac{1}{2} \left(
\begin{array}{ccc}
 1260 & -690 & -668 \\
 -690 & 394 & 341 \\
 -668 & 341 & 390 \\
\end{array}
\right) , & B & = \frac{1}{2} \left(
\begin{array}{ccc}
 2022 & -1107 & -1072 \\
 -1107 & 632 & 547 \\
 -1072 & 547 & 626 \\
\end{array}
\right) . 
\end{align*}
We have 
\begin{equation*}
\mathcal{C}(x, y) = 4 \det (A x + B y) = -23138 x^3-111903 x^2 y-180399 x y^2-96940 y^3
\end{equation*}
and $\mathcal{C}(8, -5) = 4$. Let 
\begin{equation*}
\overline{B} = 8 A - 5 B = \frac{1}{2} \left(
\begin{array}{ccc}
 -30 & 15 & 16 \\
 15 & -8 & -7 \\
 16 & -7 & -10 \\
\end{array}
\right) . 
\end{equation*}
Solving $G^{\top } \left(
\begin{array}{ccc}
 0 & 0 & 1/2 \\
 0 & -4 & 0 \\
 1/2 & 0 & 0 \\
\end{array}
\right) G = \overline{B}$, we find that  
\begin{align*}
G & = \left(
\begin{array}{ccc}
 1 & 0 & -1 \\
 -2 & 1 & 1 \\
 1 & -1 & 1 \\
\end{array}
\right) , & H & = G^{-1} = \left(
\begin{array}{ccc}
 2 & 1 & 1 \\
 3 & 2 & 1 \\
 1 & 1 & 1 \\
\end{array}
\right) . 
\end{align*}
Solving the linear Diophantine equation $\det \left(
\begin{array}{cc}
 v & 8 \\
 w & -5 \\
\end{array}
\right)  = \pm 1$, we have a solution $v = -5$, $w = 3$. Hence we put
\begin{align*}
\overline{A} = - 5 A + 3 B = \frac{1}{2} \left(
\begin{array}{ccc}
 -234 & 129 & 124 \\
 129 & -74 & -64 \\
 124 & -64 & -72 \\
\end{array}
\right) .
\end{align*}
\begin{align*}
H^{\top} \overline{A} H & = \frac{1}{2} \left(
\begin{array}{ccc}
 -14 & -29 & -1 \\
 -29 & -94 & -11 \\
 -1 & -11 & -2 \\
\end{array}
\right) , & H^{\top} \overline{B} H & = \frac{1}{2} \left(
\begin{array}{ccc}
 0 & 0 & 1 \\
 0 & -8 & 0 \\
 1 & 0 & 0 \\
\end{array}
\right) . 
\end{align*}
Finally, we find $k \in \Z$ such that the entry of $\overline{A} + k \overline{B} $ in the first row and third column is equal to $0$. In this case $k = 1$ and we obtain 
\begin{align*}
A' & = \frac{1}{2} \left(
\begin{array}{ccc}
 -14 & -29 & 0 \\
 -29 & -102 & -11 \\
 0 & -11 & -2 \\
\end{array}
\right) , & B' & = \frac{1}{2} \left(
\begin{array}{ccc}
 0 & 0 & 1 \\
 0 & -8 & 0 \\
 1 & 0 & 0 \\
\end{array}
\right) , 
\end{align*}
which provides the essential pair $[4, (- 112, -116, -51 , -11, -1) ]$ for $\EE$. 
\end{example}

To conclude this section we show how changing the polynomials defining a given quartic field affects the existence of an essential pair. Recall that the index of a field $\EE$ is the greatest common divisor of the indices of the algebraic integers in $\EE$. The following result allows for a simple test for the existence of an essential pair for a quartic field. 
\begin{theorem}\label{sectmain}
Let $\EE = \Q \left( \zeta_2 \right)$ be a quartic field of discriminant $\Delta $ and index $f$, generated by an algebraic integer $\zeta_2 $ with minimum polynomial $g_2(x)$, and let $\zeta_1 = \sum_{i=1}^4 u_{i2} \zeta_2^{i-1} \in \Z \left[ \zeta_2 \right]$ be another algebraic integer generating $\EE $ with minimum polynomial $g_1(x)$. 
\begin{enumerate}
\item If there exists $t_2 \in \Z$ satisfying $g_2 \left( t_2 \right) \equiv 0 \pmod{f^2}$, then there exists $t_1 \in \Z$ satisfying $g_1 \left( t_1 \right) \equiv 0 \pmod{f^2} $. 
\item If there exists $t_2 \in \Z$ satisfying  
\small 
\begin{align}\label{bnikftwo}
g_2 \left( t_2 \right) & \equiv 0 \pmod{f^2} , & g_2 ' \left( t_2 \right) & \equiv 0 \pmod{f} , & \gcd \left( m , f \right) & = 1 ,
\end{align}
\normalsize 
where $m = u_{22} + 2 u_{32} t_2 + 3 u_{42} t_2^2$, then there exists $t_1 \in \Z$ satisfying
\begin{align}\label{bnikfone}
g_1 \left( t_1 \right) & \equiv 0 \pmod{f^2} , & g_1 ' \left( t_1 \right) & \equiv 0 \pmod{f} .
\end{align}
\end{enumerate}
\end{theorem}

\begin{proof}
Let $\zeta_{11}$ and $\zeta_{12}$ be algebraic integers generating a quartic field $\EE = \Q \left( \zeta_{1k} \right)$, $(k = 1,2)$ and let $g_k(x) = x^4 + b_k x^3 + c_k x^2 + d_k x + e_k$ be the minimum polynomial of $\zeta_{1k}$ and let $\zeta_{1k}, \zeta_{2k}, \zeta_{3k}, \zeta_{4k}$ be the roots of $g_k(x)$. Let 
\begin{align*}
 \Xi_k & = \left[ \zeta_{ik}^{j-1} \right]_{4 \times 4} , & U & = \left[ u_{ij} \right]_{4 \times 4} , 
\end{align*}
where $\left(
\begin{array}{cccc}
 u_{11} & u_{21} & u_{31} & u_{41} \\
\end{array}
\right)^{\top } = \left(
\begin{array}{cccc}
 1 & 0 & 0 & 0 \\
\end{array}
\right)^{\top }$ and the other rational integer entries of $U$ are given by $ \Xi_1  = \Xi_2 U $. Let $P_k = \left[ p_{i+j-2}^{(k)} \right]_{4 \times 4}$ with entries defined by $P_k = \Xi_k^{\top } \Xi_k$, satisfying
\begin{align*}
p_0^{(k)} & = 4, & p_1^{(k)} & = - b_k , 
\end{align*}
\begin{equation*}
\left(
\begin{array}{ccccc}
1 & b_k & c_k & d_k & e_k 
\end{array}
\right) \left(
\begin{array}{ccccc}
 p_2^{(k)} & p_3^{(k)} & p_4^{(k)} & p_5^{(k)} & p_6^{(k)} \\
 p_1^{(k)} & p_2^{(k)} & p_3^{(k)} & p_4^{(k)} & p_5^{(k)} \\
 2 & p_1^{(k)} & p_2^{(k)} & p_3^{(k)} & p_4^{(k)} \\
 0 & 3 & p_1^{(k)} & p_2^{(k)} & p_3^{(k)} \\
 0 & 0 & 4 & p_1^{(k)} & p_2^{(k)} \\
\end{array}
\right) = [0]_{1 \times 5} . 
\end{equation*}
Then since $P_k = \Xi_k^{\top } \Xi_k$, we have 
\begin{equation*}
P_1 = U^{\top} P_2 U . 
\end{equation*}
The minimum polynomial of $\zeta_{12}$ is easily obtained in terms of the entries of the matrix $P_k$. Abbreviating the $p_j^{(k)}$ by $p_j$, let 
\begin{equation*}
S^{(k)} = \left(
\begin{array}{ccccc}
 1 & x & x^2 & x^3 & x^4 \\
 4 & p_1 & p_2 & p_3 & p_4 \\
 0 & 3 & p_1 & p_2 & p_3 \\
 0 & 0 & 2 & p_1 & p_2 \\
 0 & 0 & 0 & 1 & p_1 \\
\end{array}
\right) . 
\end{equation*}
Then 
\begin{equation*}
g_k(x) = \frac{1}{24} \det \left( S^{(k)} \right) . 
\end{equation*}
There are identities relating the entries of $U$ and $P_2$. For example, again abbreviating the $p_j^{(2)}$ by $p_j$ and $b_2$, $c_2$, $d_2$, $e_2$ by $b, c, d, e$,  
\begin{equation*}
u_{13} = u_{12}^2 - e \left( - 2 b u_{42} u_{32} + c u_{42}^2 + p_2 u_{42}^2 + u_{32}^2 + 2 u_{22} u_{42} \right) .
\end{equation*}
Letting 
\begin{equation*}
y = u_{12} + u_{22} x + u_{32} x^2 + u_{42} x^3 ,
\end{equation*}
these identities can be used to show that 
\begin{align*}
g_1 (y) & = g_2(x) q(x) , & g_1'(y) & = \frac{1}{y'(x)} \left( g_1 (y) \right)' = \frac{1}{y'(x)} \left( g_2(x) q'(x) + g_2'(x) q(x) \right) \in \Z [x] ,
\end{align*}
where $q(x)$ is a polynomial of degree $8$ and by $g_1'(y)$ we mean the replacement $x \longmapsto y$ in $g_1'(y)$. Thus if 
\small 
\begin{align*}
g_2 \left( t_2 \right) & \equiv 0 \pmod{f^2} , & g_2' \left( t_2 \right) & \equiv 0 \pmod{f} ,  
\end{align*}
\normalsize 
then, letting
\begin{eqnarray*}
t_1 & = & u_{12} + u_{22} t_2 + u_{32} t_2^2 + u_{42} t_2^3 , 
\end{eqnarray*}
$t_1$ satisfies \eqref{bnikfone}. 
\end{proof}

\section{Mapping an order to a pair of ternary quadratic forms}\label{getpair}

Given an order $\mathcal{O}$ of a quartic number field, how do we find a pair $(A, B)$ of symmetric integral matrices that parametrize the order $\mathcal{O}$? To see the difficulty in finding a pair $(A, B)$ that parametrizes a given order with a normalized basis, the problem requires solving a system of fifteen non-linear equations in twelve variables. Assuming that the $c_{ij}^{(k)}$ are known rational integers, we must solve the following system for the $a_{ij}$ and $b_{ij}$. 
\small 
\begin{align}
\nonumber c_{11}^{(1)} & = a_{13} b_{12} - a_{12} b_{13}  & c_{11}^{(2)} & = a_{11} b_{13} - a_{13} b_{11} , & c_{11}^{(3)} & = a_{12} b_{11} - a_{11} b_{12} , \\
\nonumber              &   \ + a_{23} b_{11} - a_{11} b_{23} , & & & & \\
\nonumber &   &  &   & c_{12}^{(3)} & = a_{22} b_{11} - a_{11} b_{22} , \\
\nonumber &   & c_{13}^{(2)} & =  a_{11} b_{33} - a_{33} b_{11} , & c_{13}^{(3)} & = a_{23} b_{11} - a_{11} b_{23} , \\
\nonumber c_{22}^{(1)} & = a_{23} b_{22} - a_{22} b_{23} , & c_{22}^{(2)} & = a_{12} b_{23} - a_{23} b_{12}  & c_{22}^{(3)} & = a_{22} b_{12} - a_{12} b_{22} , \\
\label{tonmgdf}             &                                   &              & \ + a_{22} b_{13} - a_{13} b_{22}  , & & \\ 
\nonumber c_{23}^{(1)} & = a_{33} b_{22} - a_{22} b_{33} , & c_{23}^{(2)} & = a_{12} b_{33} - a_{33} b_{12} , & c_{23}^{(3)} & = a_{22} b_{13} - a_{13} b_{22} , \\
\nonumber c_{33}^{(1)} & = a_{33} b_{23} - a_{23} b_{33} , & c_{33}^{(2)} & = a_{13} b_{33} - a_{33} b_{13} , & c_{33}^{(3)} & = a_{12} b_{33} - a_{33} b_{12} \\
\nonumber              &                                   &              &                                 &              & \ + a_{23} b_{13} - a_{13} b_{23} .
\end{align}
\normalsize

The following result shows that much fewer than fifteen scalar equations are required to be solved for the $a_{ij}$, $b_{ij} \in \Z$ in order to recover the twelve parameters of an order of a quartic field. This result leads to our algorithm for solving the problem of finding a pair of ternary quadratic forms that parametrizes a given order of a quartic field. 

\begin{lemma}\label{nineeqns}
Let $\mathcal{O}$ be an order of a quartic field with normalized basis $\left\{ 1, \omega_1, \omega_2, \omega_3 \right\}$ with coefficients $c_{ij}^{(k)}$ given by $\omega_i \omega_j = \sum_{k = 0}^{3} c_{ij}^{(k)} \omega_k$, $\omega_0 = 1$. Let 
\begin{align}\label{defell}
\ell & = c_{13}^{(3)} - c_{11}^{(1)} , & m & = c_{33}^{(3)} - c_{23}^{(2)} , & n & = c_{22}^{(2)} - c_{23}^{(3)} ,
\end{align}
let the column matrices $C_j$, the concatenation $Z$ of the $C_j$, the row matrices $R_j$, and the concatenation $W$ of the $R_j$, be given by 
\small 
\begin{align*}
Z & = \left(
\begin{array}{cccccc}
 C_{1} & C_{2} & C_{3} & C_{4} & C_{5} & C_{6} \\
\end{array}
\right) = \left(
\begin{array}{cccccc}
 a_{11} & a_{12} & a_{13} & a_{22} & a_{23} & a_{33} \\
 b_{11} & b_{12} & b_{13} & b_{22} & b_{23} & b_{33} \\
\end{array}
\right) , \\
 W & = \left[ w_{ij} \right] = \left(
\begin{array}{c}
 R_1  \\
 R_2  \\
 R_3  \\
 R_4  \\
\end{array}
\right) = \left(
\begin{array}{cccccc}
 \ell & - c_{11}^{(3)} & c_{11}^{(2)} & - c_{12}^{(3)} & - c_{13}^{(3)} & c_{13}^{(2)} \\
 c_{11}^{(2)} & \ell & 0 & c_{23}^{(3)} & m & - c_{33}^{(2)} \\
 - c_{11}^{(3)} & 0 & - \ell &  c_{22}^{(3)} & - n & - c_{23}^{(2)} \\
 - c_{12}^{(3)} & - c_{22}^{(3)} & -c_{23}^{(3)} & c_{11}^{(2)} & c_{22}^{(1)} & c_{23}^{(1)} \\
\end{array}
\right) .
\end{align*}
\normalsize 
Let 
\begin{align*}
X_1 & = \left(
\begin{array}{cc}
 C_{2} & C_{3} \\
\end{array}
\right) , & X_2 & = \left(
\begin{array}{cc}
 C_{1} & C_{3} \\
\end{array}
\right) , & X_3 & = \left(
\begin{array}{cc}
 C_{1} & C_{2} \\
\end{array}
\right) , & X_4 & = \left(
\begin{array}{cc}
 C_{1} & C_{4} \\
\end{array}
\right) , \\
 Y_1 & = \left(
\begin{array}{c}
 R_2  \\
 - R_3  \\
\end{array}
\right) , & Y_k & = \left(
\begin{array}{c}
 R_k  \\
 R_1  \\
\end{array}
\right) , & & & & 
\end{align*}
for $k = 2, 3, 4$ and $x_j = \det \left( X_j \right) $ for each $j = 1, 2, 3, 4$. The $a_{11}, a_{12}, a_{13}, a_{22}, a_{23}, a_{33}$, $b_{11}, \dots , b_{33}$ are twelve integers satisfying one of the four systems of equations (for $t = 1, 2, 3, 4$)
\begin{align}
\label{foureqns} X_t Y_t & = x_t Z , & x_t & = w_{t1}  
\end{align}
if and only if the fifteen equations \eqref{tonmgdf} are satisfied by the $a_{ij}, b_{ij}$, where 
\begin{enumerate}
\item $\ell \not = 0$ whenever $t = 1$.
\item $\ell = 0$ and $c_{11}^{(2)} \not = 0$ whenever $t = 2$. 
\item $\ell = 0 = c_{11}^{(2)}$ and $c_{11}^{(3)} \not = 0$ whenever $t = 3$.
\item $\ell = 0 = c_{11}^{(2)} = c_{11}^{(3)}$ and $c_{12}^{(3)} \not = 0$ whenever $t = 4$. 
\end{enumerate}  
\end{lemma}

\begin{proof}
Verifying that \eqref{foureqns} holds requires using the following identities, which can be obtained by eliminating all twelve of the $a_{ij}$ and $b_{ij}$ from the equations in \eqref{tonmgdf}. 
\small 
\begin{align}
\label{propbasis} 
\ell c_{12}^{(3)} & = c_{11}^{(2)} c_{22}^{(3)} + c_{11}^{(3)} c_{23}^{(3)} , & \ell c_{13}^{(2)} & = c_{11}^{(2)} c_{23}^{(2)} + c_{11}^{(3)} c_{33}^{(2)} , & \ell c_{23}^{(1)} & = c_{22}^{(3)} c_{33}^{(2)} - c_{23}^{(2)} c_{23}^{(3)} , \\
\nonumber \ell c_{13}^{(3)} & = c_{11}^{(3)} m - c_{11}^{(2)} n , & \ell c_{22}^{(1)} & = - c_{22}^{(3)} m - c_{23}^{(3)} n , & \ell c_{33}^{(1)} & = - c_{23}^{(2)} m -  c_{33}^{(2)} n , \\
\label{case2eqns} c_{11}^{(2)} c_{22}^{(1)} & = c_{13}^{(3)} c_{23}^{(3)} - m c_{12}^{(3)} , & c_{11}^{(2)} c_{23}^{(1)} & = c_{12}^{(3)} c_{33}^{(2)} - c_{13}^{(2)} c_{23}^{(3)} , & c_{11}^{(2)} c_{33}^{(1)} & = c_{13}^{(3)} c_{33}^{(2)} - m c_{13}^{(2)} , \\
\label{case3eqns} c_{11}^{(3)} c_{23}^{(1)} & = c_{13}^{(2)} c_{22}^{(3)} - c_{12}^{(3)} c_{23}^{(2)} , & c_{11}^{(3)} c_{33}^{(1)} & = - c_{13}^{(3)} c_{23}^{(2)} - n c_{13}^{(2)} , & c_{11}^{(3)} c_{22}^{(1)} & = - c_{22}^{(3)} c_{13}^{(3)} - n c_{12}^{(3)}  \\
\label{case4eqns} c_{33}^{(1)} c_{12}^{(3)} & = c_{13}^{(3)}c_{23}^{(1)} + c_{13}^{(2)} c_{22}^{(1)} . & & & &  
\end{align}
\normalsize 
We first consider the case where $t = 1$, in other words $x_1 = \ell \not = 0$. First assume that the normalized basis for $\mathcal{O}$ is parametrized by the $a_{ij}, b_{ij}$ and hence the fifteen equations \eqref{tonmgdf} are satisfied by the $a_{ij}, b_{ij}$. For this case, we will only require equations from (\ref{propbasis}).
Since we assume that $\ell \not= 0$, we obtain the correct $c_{ij}^{(k)}$. It is important to note that the equations \eqref{propbasis}, \eqref{case2eqns}, \eqref{case3eqns}, \eqref{case4eqns} are properties of a normalized basis. 

Conversely, assume that \eqref{foureqns} holds for some twelve integers $a_{ij}, b_{ij}$. Then, solving this matrix equation together with $\ell = \det \left( X_1 \right) \ ( \not= 0)$ for the $c_{ij}^{(k)}$ gives
\begin{equation*}
Y_1 = \left(
\begin{array}{cccccc}
	x_2 & x_1 & 0 & a_{22} b_{13}-a_{13} b_{22} & a_{23} b_{13}-a_{13} b_{23} & a_{33} b_{13}-a_{13} b_{33} \\
	-x_3 & 0 & x_1 & a_{12} b_{22}-a_{22} b_{12} & a_{12} b_{23}-a_{23} b_{12} & a_{12} b_{33}-a_{33} b_{12} \\
\end{array}
\right) , 
\end{equation*}
while by definition we have 
\begin{equation*}
Y_1 = \left(
\begin{array}{cccccc}
c_{11}^{(2)} & \ell & 0 & c_{23}^{(3)} & m & - c_{33}^{(2)} \\
c_{11}^{(3)} & 0 & \ell &  -c_{22}^{(3)} &  n & c_{23}^{(2)} \\
\end{array}
\right) .
\end{equation*}
By equating these matrices we obtain $c_{11}^{(2)} = a_{11} b_{13} - a_{13} b_{11}$ from the $(1,1)$-entry, while subsequent entries yield formulas for $c_{11}^{(3)}$, $c_{23}^{(3)}$, $c_{22}^{(3)}$, $m$, $n$, $c_{33}^{(2)}$, and $c_{23}^{(2)}$. Each of these formulas coincides with those given in \eqref{tonmgdf}. Since $c_{33}^{(3)} = m + c_{23}^{(2)}$ and $c_{22}^{(2)} = n + c_{23}^{(3)} $, we also obtain formulas for $c_{33}^{(3)}$ and $c_{22}^{(2)}$ that agree with \eqref{tonmgdf}. Since $\ell \not= 0$, we can use \eqref{propbasis} to obtain formulas for the remaining $c_{ij}^{(k)}$ including $c_{12}^{(3)}$. These formulas agree with \eqref{tonmgdf}. Also properties of a normalized basis are $c_{12}^{(1)} = c_{12}^{(2)} = c_{13}^{(1)} = 0$. Hence all twenty-four $c_{ij}^{(k)}$ may be correctly recovered from \eqref{foureqns} in the case that $\ell \not= 0$. 

The remaining three cases $t = 2, 3, 4$ are proved similarly. For example, in Case 2 if we assume that \eqref{foureqns} holds for some twelve integers $a_{ij}, b_{ij}$, then, solving this matrix equation together with $c_{11}^{(2)} = \det \left( X_2 \right)  \ ( \not= 0)$ for the $c_{ij}^{(k)}$ gives
\begin{equation*}
Y_2 = \left(
\begin{array}{cccccc}
 x_2 & x_1 & 0 & a_{22} b_{13}-a_{13} b_{22} & a_{23} b_{13}-a_{13} b_{23} & a_{33} b_{13}-a_{13} b_{33} \\
 0 & x_3 & x_2 & x_4 & a_{11} b_{23}-a_{23} b_{11} & a_{11} b_{33}-a_{33} b_{11} \\
\end{array}
\right) , 
\end{equation*}
while by definition we have 
\begin{equation*}
Y_2 = \left(
\begin{array}{cccccc}
 c_{11}^{(2)} & \ell & 0 & c_{23}^{(3)} & m & - c_{33}^{(2)} \\
 \ell & - c_{11}^{(3)} & c_{11}^{(2)} & - c_{12}^{(3)} & - c_{13}^{(3)} & c_{13}^{(2)} \\
\end{array}
\right) . 
\end{equation*}
This immediately provides expressions for $c_{11}^{(1)}$, $c_{11}^{(2)}$, $c_{11}^{(3)}$, $c_{12}^{(3)}$, $c_{13}^{(2)}$, $c_{13}^{(3)}$, $c_{23}^{(3)}$, $c_{33}^{(2)}$ that agree with \eqref{tonmgdf}. Since $c_{11}^{(2)} c_{22}^{(1)} = c_{13}^{(3)} c_{23}^{(3)} - m c_{12}^{(3)}$ and $m = a_{23} b_{13} - a_{13} b_{23}$, we obtain an expression for $c_{22}^{(1)}$ that agrees with \eqref{tonmgdf}. Likewise, since $\ell = 0$, \eqref{propbasis} gives $c_{11}^{(2)} n = m c_{11}^{(3)}$, $c_{11}^{(2)} c_{22}^{(3)} = - c_{11}^{(3)} c_{23}^{(3)}$, and $c_{11}^{(2)} c_{23}^{(2)} = - c_{11}^{(3)} c_{33}^{(2)}$. Since $c_{22}^{(2)} = n + c_{23}^{(3)}$ and $c_{33}^{(3)} = m + c_{23}^{(2)}$, we obtain expressions for $c_{22}^{(2)}$, $c_{22}^{(3)}$, $c_{23}^{(2)}$, and $c_{33}^{(3)}$ that agree with \eqref{tonmgdf}. Finally, $c_{11}^{(2)} c_{23}^{(1)} = c_{12}^{(3)} c_{33}^{(2)} - c_{13}^{(2)} c_{23}^{(3)}$ and $c_{11}^{(2)} c_{33}^{(1)} = c_{13}^{(3)} c_{33}^{(2)} - m c_{13}^{(2)}$ so we are able to solve for $c_{23}^{(1)} $ and $c_{33}^{(1)} $ in agreement with \eqref{tonmgdf}.

Case 3 proceeds as follows: Given that \eqref{foureqns} holds for some twelve integers $a_{ij}, b_{ij}$, solving this matrix equation together with $- c_{11}^{(3)} = \det \left( X_3 \right) \ ( \not= 0)$ for the $c_{ij}^{(k)}$ gives
\begin{equation*}
Y_3 = \left(
\begin{array}{cccccc}
x_3 & 0 & - x_1 & a_{22} b_{12}-a_{12} b_{22} & a_{23} b_{12}-a_{12} b_{23} & a_{33} b_{12}-a_{12} b_{33} \\
0 & x_3 & x_2 & x_4 & a_{11} b_{23}-a_{23} b_{11} & a_{11} b_{33}-a_{33} b_{11} \\
\end{array}
\right) ,
\end{equation*}
while by definition we have 
\begin{equation*}
Y_3 = \left(
\begin{array}{cccccc}
 - c_{11}^{(3)} & 0 & - \ell &  c_{22}^{(3)} & - n & - c_{23}^{(2)} \\
 \ell & - c_{11}^{(3)} & c_{11}^{(2)} & - c_{12}^{(3)} & - c_{13}^{(3)} & c_{13}^{(2)} \\
\end{array}
\right) . 
\end{equation*}
This gives us seven equations for the variables $- x_3  = c_{11}^{(3)}$, $c_{12}^{(3)}$, $c_{13}^{(2)}$, $c_{13}^{(3)}$, $c_{22}^{(3)}$, $c_{23}^{(2)}$, $n$ which match with \eqref{tonmgdf}. Note that this equation also gives us the correct identity for $\ell$ since $\ell = 0$ and hence we also get $c_{11}^{(1)}$. From the $(2,3)$-th entry, we get $a_{11} b_{13}-a_{13} b_{11} = 0 = c_{11}^{(2)}$ since we have assumed that $t = 3$.

To obtain the remaining identities for the $c_{ij}^{(k)}$ we can use the equations \eqref{case3eqns} which give the equations for $c_{23}^{(1)}$, $c_{33}^{(1)}$ and $c_{22}^{(1)}$. The system \eqref{propbasis} and the fact that $c_{11}^{(2)} = 0, c_{11}^{(3)}\neq 0$ gives us 
$$\ell c_{12}^{(3)} = c_{11}^{(3)}c_{23}^{(3)}, \qquad \ell c_{13}^{(2)} = c_{11}^{(3)}c_{33}^{(2)},
 \qquad  \ell  c_{13}^{(3)}  =  c_{11}^{(3)}  m.$$ From these we obtain the correct equations for $c_{23}^{(3)}$ and $c_{33}^{(2)}$. Moreover, since we know $\ell = 0$, we get that $c_{23}^{(3)} = c_{33}^{(2)} = m = 0$. One may immediately deduce that $c_{22}^{(2)} = n$ and  $c_{33}^{(3)} = 0$ and both satisfy the desired formulas. 
 
Finally to prove Case 4, given that \eqref{foureqns} holds for some twelve integers $a_{ij}, b_{ij}$, solving this matrix equation together with $- c_{12}^{(3)}= \det \left( X_4 \right) \ ( \not= 0)$ for the $c_{ij}^{(k)}$ gives
\small 
\begin{equation*}
Y_4 =
\left(
\begin{array}{cccccc}
 x_4 & a_{12} b_{22}-a_{22} b_{12} & a_{13} b_{22}-a_{22} b_{13} & 0 & a_{23} b_{22}-a_{22} b_{23} & a_{33} b_{22}-a_{22} b_{33} \\
 0 & x_3 & x_2 & x_4 & a_{11} b_{23}-a_{23} b_{11} & a_{11} b_{33}-a_{33} b_{11} \\
\end{array}
\right)
, 
\end{equation*} 
\normalsize   
while by definition we have 
\begin{equation*}
Y_4 = \left(
\begin{array}{cccccc}
 - c_{12}^{(3)} & - c_{22}^{(3)} & -c_{23}^{(3)} & c_{11}^{(2)} & c_{22}^{(1)} & c_{23}^{(1)} \\ 
 \ell & - c_{11}^{(3)} & c_{11}^{(2)} & - c_{12}^{(3)} & - c_{13}^{(3)} & c_{13}^{(2)} \\
\end{array}
\right) . 
\end{equation*}
Equations for $c_{11}^{(2)},c_{11}^{(3)}, c_{12}^{(3)}, c_{13}^{(2)}, c_{13}^{(3)}, c_{22}^{(1)}, c_{22}^{(3)}, c_{23}^{(1)}$ and $c_{23}^{(3)}$ can be immediately verified. Using \eqref{case4eqns}, we can obtain the correct equations for $c_{23}^{(2)},c_{33}^{(1)}$ and $c_{33}^{(2)}$. The first equation of \eqref{propbasis} provides a correct equation for $\ell$, and the identity $\ell = c_{13}^{(3)} - c_{11}^{(1)}$ gives the formula for $c_{11}^{(1)}$. Similarly, we can use \eqref{case3eqns} to get the appropriate formula for $n$, and use the identity $n = c_{22}^{(2)}- c_{23}^{(3)}$ to obtain the correct formula for $c_{22}^{(2)}$. The formula for $c_{33}^{(3)}$ may be obtained in a similar fashion. This completes the proof of Case 4.
\end{proof}

There are no more than four cases in Lemma \ref{nineeqns} since if $c_{11}^{(1)} - c_{13}^{(3)} = 0 = c_{11}^{(2)} = c_{11}^{(3)} = c_{12}^{(3)} $, then $\mathcal{O}$ is not an order of a quartic field because the $\mathcal{O}$ have no divisors of zero and the assumption would force one of $\omega_1$ or $\omega_2$ to be zero so that $\mathcal{O}$ would have a zero discriminant. 

\begin{algorithm}\label{secondalgparams}
Given a normalized basis of an order $\mathcal{O}$ of a quartic field, compute a pair of ternary quadratic forms that parametrizes $\mathcal{O}$. 
\begin{enumerate}
\item Compute $t \in \{ 1, 2, 3, 4 \}$ according to the following rules, where $\ell $ is given by \eqref{defell}: If $\ell \not= 0$, then $t = 1$; if $\ell = 0$, $c_{11}^{(2)} \not = 0$, then $t = 2$; if $\ell = 0 = c_{11}^{(2)}$, $c_{11}^{(3)} \not = 0$, then $t = 3$; if $\ell = 0 = c_{11}^{(2)} = c_{11}^{(3)}$, $c_{12}^{(3)} \not = 0$, then $t = 4$. 
\item Construct the matrix $M_t$ depending on $t = 1, 2, 3, 4$, where 
\small 
\begin{eqnarray}
\label{themt} M_1 & = & \left(
\begin{array}{cccccc}
 - \ell & c_{11}^{(2)} & c_{11}^{(3)} & 0 & 0 & 0 \\
 0 & c_{23}^{(3)} & - c_{22}^{(3)} & - \ell & 0 & 0 \\
 0 & m & n & 0 & - \ell & 0 \\
 0 & - c_{33}^{(2)} & c_{23}^{(2)} & 0 & 0 & - \ell \\
\end{array}
\right) , \\
\nonumber M_2 & = & \left(
\begin{array}{cccccc}
 0 & -c_{11}^{(2)} & -c_{11}^{(3)} & 0 & 0 & 0 \\
 c_{23}^{(3)} & 0 & -c_{12}^{(3)} & -c_{11}^{(2)} & 0 & 0 \\
 m & 0 & -c_{13}^{(3)} & 0 & -c_{11}^{(2)} & 0 \\
 -c_{33}^{(2)} & 0 & c_{13}^{(2)} & 0 & 0 & -c_{11}^{(2)} \\
\end{array}
\right) , \\
\nonumber M_3 & = & \left(
\begin{array}{cccccc}
 0 & 0 & c_{11}^{(3)} & 0 & 0 & 0 \\
 c_{22}^{(3)} & -c_{12}^{(3)} & 0 & c_{11}^{(3)} & 0 & 0 \\
 - n & - c_{13}^{(3)} & 0 & 0 & c_{11}^{(3)} & 0 \\
 - c_{23}^{(2)} & c_{13}^{(2)} & 0 & 0 & 0 & c_{11}^{(3)} \\
\end{array}
\right) , \\
\nonumber M_4 & = & \left(
\begin{array}{cccccc}
 -c_{22}^{(3)} & c_{12}^{(3)} & 0 & 0 & 0 & 0 \\
 -c_{23}^{(3)} & 0 & c_{12}^{(3)} & 0 & 0 & 0 \\
 c_{22}^{(1)} & 0 & 0 & -c_{13}^{(3)} & c_{12}^{(3)} & 0 \\
 c_{23}^{(1)} & 0 & 0 & c_{13}^{(2)} & 0 & c_{12}^{(3)} \\
\end{array}
\right) . 
\end{eqnarray}
\normalsize 
\item Compute the row-style Hermite normal form $(U \mid V)$ of $\left( M_t^{\top} \mid I_6 \right)$, where the entries of $V \in \text{GL}_6(\Z )$ are denoted by $v_{ij}$.
\item Define the integer $k$ as follows:
\begin{align}\label{kfiveone}
k_{} & = 
  \begin{cases}
  \frac{\ell }{\left( v_{52} v_{63} - v_{53} v_{62} \right)} & \text{if $t = 1$,} \\
  \frac{c_{11}^{(2)}}{ v_{51} v_{63} }  & \text{if $t = 2$,} \\
  \frac{- c_{11}^{(3)}}{ v_{51} v_{62} } & \text{if $t = 3$,} \\
  \frac{- c_{12}^{(3)}}{ v_{51} v_{64} } & \text{if $t = 4$.}
  \end{cases} 
\end{align}
\item Output a pair of ternary quadratic forms
\begin{eqnarray*}
\mathcal{Q}_A & = & \left( a_{11}, a_{12}, a_{13}, a_{22}, a_{23}, a_{33} \right) = k_{} \left( v_{51}, v_{52}, v_{53}, v_{54}, v_{55}, v_{56} \right) , \\
\mathcal{Q}_B & = & \left( b_{11}, b_{12}, b_{13}, b_{22}, b_{23}, b_{33} \right) = \left( 0, v_{62}, v_{63}, v_{64}, v_{65}, v_{66} \right)  
\end{eqnarray*}
that parametrizes the order $\mathcal{O}$.
\end{enumerate}
\end{algorithm}

The main approach of our proof is to use Lemma \ref{nineeqns} and the method set out in Gilbert and Pathria \cite{GilbPath} for solving systems of linear Diophantine equations. Some of these details are apparent in Example \ref{oigfdas} below. When we use Algorithm \ref{secondalgparams} to find the $a_{ij}, b_{ij}$, Lemma \ref{nineeqns} guarantees that all twenty-four identities for the $c_{ij}^{(k)}$ are satisfied and the order is parametrized by a pair of ternary quadratic forms determined by the twelve integers $a_{ij}, b_{ij}$. 

\begin{proof}
Let $[0]$ denote the $4 \times 2$ matrix with zero entries, let 
\begin{equation*}
\left( w_{11}, w_{21}, w_{31}, w_{41} \right) = \left( \ell , c_{11}^{(2)}, - c_{11}^{(3)}, - c_{12}^{(3)} \right) , 
\end{equation*}
and let 
\small 
\begin{align*}
x_1 & = a_{12} b_{13} - a_{13} b_{12} , & x_2 & = a_{11} b_{13} - a_{13} b_{11} , & x_3 & = a_{11} b_{12} - a_{12} b_{11}, & x_4 & = a_{11} b_{22} - a_{22} b_{11} ,
\end{align*}
where the $a_{ij}$ and $b_{ij}$ are to be treated as indeterminants rather than any original parameters that may have been used to obtain the order $\mathcal{O}$.
\normalsize 
\begin{itemize}
\item Case $t$ : \ 
\begin{enumerate}
\item We must solve for $a_{ij}, b_{ij}$ the linear system
\begin{equation}\label{vareksdf}
M_t Z^T = [0] ,
\end{equation}
using the method set out in Gilbert and Pathria \cite{GilbPath}, where $Z$ is the matrix containing the $a_{ij}, b_{ij}$ given in Lemma \ref{nineeqns}. This equation is equivalent to $X_t Y_t = w_{t1} Z$ in Lemma \ref{nineeqns}. Recall from \cite{GilbPath} that we must unimodular row-reduce the matrix $\left( M_t^{\top} \mid I_6 \right)$ to $(U \mid V)$, where $U = \left[ u_{ij} \right]$ is in row echelon form and $V = \left[ v_{ij} \right] \in \text{GL}_6(\Z )$ with $v_{61} = 0$ since this action takes the row-style Hermite norm form of $\left( M_t^{\top} \mid I_6 \right)$. We then solve $U^{\top } K = 0$ for $K$ with entries in $\Z $. 

All solutions to \eqref{vareksdf} are of the form $Z^{\top } = V^{\top } K$. Furthermore, the top four rows of $K$ will have zero entries. This occurs since the matrix $U$ is in row-style Hermite normal form, meaning any rows of zeros will occur at the bottom. Since $M_t^{\top}$ has dimension four (over $\mathbb{Q}$) and $M$ has integer entries, $U$ is of the same dimension and thus has precisely two rows of zeros; row-style Hermite normal form implies that $U^{\top}$ must have two columns of zeros on the right, which forces $K$ to have solutions with the first four rows of $K$ having zero entries. The solution $K$ will not be unique; we express the result in terms of four integers $k_{51}, k_{52}, k_{61}, k_{62}$. Due to rows of zeros in $K$, we have 
\begin{equation*}
Z = \left(
\begin{array}{cc}
 k_{51} & k_{61} \\
 k_{52} & k_{62} \\
\end{array}
\right) \left(
\begin{array}{cccccc}
 v_{51} & v_{52} & v_{53} & v_{54} & v_{55} & v_{56} \\
 v_{61} & v_{62} & v_{63} & v_{64} & v_{65} & v_{66} \\
\end{array}
\right) .
\end{equation*}
\item Following Lemma \ref{nineeqns}, we must choose the $k_{51}, k_{52}, k_{61}, k_{62} \in \Z$ so that $w_{t1} = x_t$. This is easy; we can choose $k_{52} = 0$, $k_{61} = 0$, $k_{62} = 1$ and $k_{51}$ as appropriate. By the definitions of $w_{t1}$ and $x_t$, since the first four rows of $K$ have zero entries, and $v_{61} = 0$, we have 
\begin{eqnarray*}
x_1 & = & \left( k_{51} k_{62} - k_{52} k_{61} \right) \left( v_{52} v_{63} - v_{53} v_{62} \right) = \ell , \\
x_2 & = & \left( k_{51} k_{62} - k_{52} k_{61} \right) v_{51} v_{63} = c_{11}^{(2)} , \\
x_3 & = & \left( k_{51} k_{62} - k_{52} k_{61} \right) v_{51} v_{62} = - c_{11}^{(3)} , \\
x_4 & = & \left( k_{51} k_{62} - k_{52} k_{61} \right) v_{51} v_{64}  = - c_{12}^{(3)} .
\end{eqnarray*} 
Bhargava \cite{Bhargavacomp4} showed that there exists at least one solution set \\
$\{a_{11}, a_{12}, \ldots b_{33} \} \in \Z$ which corresponds to a pair of ternary quadratic forms parametrizing $\mathcal{O}$. Hence for example, when $\ell \not = 0$, we must have $\left( v_{52} v_{63} - v_{53} v_{62} \right) \mid \ell $ and more generally, a solution to $w_{t1} = x_t$ exists. Since we can put $k_{52} = 0$, $k_{61} = 0$, $k_{62} = 1$, this means that we may choose the $k_{51} = k$ as given by \eqref{kfiveone}. Finally, this choice of $k_{ij}$ provides the pair of ternary quadratic forms
\begin{eqnarray*}
\mathcal{Q}_A & = & k_{} \left( v_{51}, v_{52}, v_{53}, v_{54}, v_{55}, v_{56} \right) , \\
\mathcal{Q}_B & = & \left( 0, v_{62}, v_{63}, v_{64}, v_{65}, v_{66} \right)  
\end{eqnarray*}
that parametrizes the order $\mathcal{O}$, where $k_{51}$ is given by \eqref{kfiveone}. 
\end{enumerate}
\end{itemize}
\end{proof}

\begin{example}\label{oigfdas}
Again consider the order parametrized by $\mathcal{Q}_A = (2, -5, 3, 3, 1, 3)$ and $\mathcal{Q}_B = (0, -4, 3, -3, 1, -3)$. Recall that we obtained an order $\mathcal{O}$ where $\ell = -3$ and hence $t = 1$. Following Algorithm \ref{secondalgparams} we calculate the row-style Hermite normal form $(U \mid V)$ of $\left( M_t^{\top} \mid I_6 \right)$ and obtain 
\begin{align*}
U^{\top} & = \left(
\begin{array}{cccccc}
 1 & 0 & 0 & 0 & 0 & 0 \\
 0 & 3 & 0 & 0 & 0 & 0 \\
 1 & 0 & 3 & 0 & 0 & 0 \\
 0 & 0 & 0 & 3 & 0 & 0 \\
\end{array}
\right) , & V^{\top } & = \left(
\begin{array}{cccccc}
 1 & 0 & 0 & 0 & 2 & 0 \\
 1 & 0 & 0 & 0 & 3 & 4 \\
 -1 & 0 & 0 & 0 & -3 & -3 \\
 3 & 1 & 0 & 0 & 9 & 3 \\
 0 & 0 & 1 & 0 & -1 & -1 \\
 3 & 0 & 0 & 1 & 9 & 3 \\
\end{array}
\right) .
\end{align*}
Solutions $K$ to $U^{\top } K = 0$ are of the form 
\begin{equation*}
K = \left(
\begin{array}{cc}
 0 & 0 \\
 0 & 0 \\
 0 & 0 \\
 0 & 0 \\
 k_{51} & k_{52} \\
 k_{61} & k_{62} \\
\end{array}
\right) ,
\end{equation*}
for any $k_{51}, k_{52} , k_{61}, k_{62} \in \Z$. All solutions to $M_t Z^{\top } = 0$ are of the form $Z^{\top } = V^{\top } K$. It follows that 
\small
\begin{equation*}
Z = \left(
\begin{array}{cccccc}
 2 k_{51} & 3 k_{51}+4 k_{61} & -3 k_{51}-3 k_{61} & 9 k_{51}+3 k_{61} & -k_{51}-k_{61} & 9 k_{51}+3 k_{61} \\
 2 k_{52} & 3 k_{52}+4 k_{62} & -3 k_{52}-3 k_{62} & 9 k_{52}+3 k_{62} & -k_{52}-k_{62} & 9 k_{52}+3 k_{62} \\
\end{array}
\right) , 
\end{equation*}
\normalsize 
provided that $\ell = a_{12} b_{13} - a_{13} b_{12}$. We have $$a_{12} b_{13} - a_{13} b_{12} = 3 \left( k_{51} k_{62} - k_{52} k_{61} \right) $$ so we must choose the $k_{ij}$ so that $k_{51} k_{62} - k_{52} k_{61} = -1$. If we  choose 
\begin{equation*}
\left(
\begin{array}{cc}
 k_{51} & k_{52} \\
 k_{61} & k_{62} \\
\end{array}
\right) = \left(
\begin{array}{cc}
 1 & 0 \\
 -2 & -1 \\
\end{array}
\right) ,
\end{equation*}
then we recover the original twelve parameters used to construct $\mathcal{O}$. The solution given by Algorithm \ref{secondalgparams} puts 
\begin{equation*}
\left(
\begin{array}{cc}
 k_{51} & k_{52} \\
 k_{61} & k_{62} \\
\end{array}
\right) = \left(
\begin{array}{cc}
 -1 & 0 \\
 0 & 1 \\
\end{array}
\right) ,
\end{equation*}
and 
\begin{align*}
\mathcal{Q}_A & = (-2, -3, 3, -9, 1, -9), & \mathcal{Q}_B & = (0, 4, -3, 3, -1, 3) . 
\end{align*}
\end{example}

\begin{example}\label{secondeg}
Let $\EE$ be the unique quartic field of discriminant $1424$, a field for which there is no essential pair. An arithmetic matrix for a normalized integral basis $\left\{ 1, \omega_1, \omega_2, \omega_3 \right\}$ of the ring of integers $\mathcal{O}$ of $\EE$ is given by 
\footnotesize  
\begin{equation*}
N_{\mathcal{O}}^{(\alpha )} = \left(
\begin{array}{cccc}
 u & -x-4 z & -2 y & -4 x-6 z \\
 x & u-x & -2 z & 2 z-2 y \\
 y & 2 x+3 z & u+2 y & 3 x+2 z \\
 z & y-z & x+2 z & u-x+2 y-2 z \\
\end{array}
\right) ,
\end{equation*}
\normalsize 
from which we obtain the coefficients $c_{ij}^{(k)}$ in agreement with \eqref{quartpars}. The matrix $N_{\mathcal{O}}^{(\alpha )}$ facilitates arithmetic of the elements $u + x \omega_1 + y \omega_2 + z \omega_3$ of $\mathcal{O}$. We have $\ell = 0$ and hence $t = 2$. Following Algorithm \ref{secondalgparams} we calculate the row-style Hermite normal form $(U \mid V)$ of $\left( M_t^{\top} \mid I_6 \right)$ and obtain 
\footnotesize  
\begin{align*}
U^{\top} & = \left(
\begin{array}{cccccc}
 2 & 0 & 0 & 0 & 0 & 0 \\
 0 & 1 & 0 & 0 & 0 & 0 \\
 0 & 1 & 2 & 0 & 0 & 0 \\
 0 & 1 & 0 & 2 & 0 & 0 \\
\end{array}
\right) , & V^{\top } & = \left(
\begin{array}{cccccc}
 0 & 0 & 0 & 0 & 1 & 0 \\
 -1 & 0 & 0 & 0 & 0 & 0 \\
 0 & 1 & 0 & 0 & 0 & 2 \\
 0 & -1 & 0 & 0 & 1 & -1 \\
 0 & 0 & -1 & 0 & -1 & 1 \\
 0 & 1 & 0 & -1 & -1 & 3 \\
\end{array}
\right) .
\end{align*}
\normalsize 
Since $k = 1$, the required parameters are apparent in the right-most columns of the matrix $V^{\top }$,
\begin{align*}
\mathcal{Q}_A & = (1, 0, 0, 1, -1, -1), & \mathcal{Q}_B & = (0, 0, 2, -1, 1, 3) . 
\end{align*}
The binary cubic form $4 \det (A x + B y) = (-5, 18, -17, 4)$ has discriminant $1424$.
\end{example}

The following table lists pairs of ternary quadratic forms that parametrize the ring of integers of quartic number fields of small positive discriminant. These were found by computing an essential pair when they exist. Where there is no essential pair, Algorithm \ref{secondalgparams} was used.

\tiny
\begin{center}
\begin{longtable}{|c|c|c|c|c|}
\caption{Pairs of TQFs parametrizing the maximal order of quartic fields of positive discriminant $\Delta $. From left to right we have the field discriminant, an essential pair, $\left( a_{11}, a_{12}, a_{13}, a_{22}, a_{23}, a_{33} \right)$, $\left( b_{11}, b_{12}, b_{13}, b_{22}, b_{23}, b_{33} \right)$, and $4 \det (A x + B y)$.} \label{quarticindf} \\

\hline \multicolumn{1}{|c|}{$\Delta$} & \multicolumn{1}{c|}{ $\left[ f, \left( a, b, c, d, e \right) \right]$} & \multicolumn{1}{c|}{ \ $\mathcal{Q}_A$ \ } & \multicolumn{1}{c|}{\ $\mathcal{Q}_B$ \ } & \multicolumn{1}{c|}{ \ $4 \det (A x + B y)$ } \\ \hline
\endfirsthead

\multicolumn{4}{c}{\tablename\ \thetable{} -- continued} \\
\hline \multicolumn{1}{|c|}{$\Delta$} & \multicolumn{1}{c|}{$\left[ f, \left( a, b, c, d, e \right) \right]$} & \multicolumn{1}{c|}{ \ $\mathcal{Q}_A$ \ } & \multicolumn{1}{c|}{ \ $\mathcal{Q}_B$ \ } & \multicolumn{1}{c|}{ \ $4 \det (A x + B y)$ } \\ \hline
\endhead

\hline \multicolumn{5}{|c|}{{Continued}} \\ \hline
\endfoot

\hline 
\endlastfoot

 117 & [1, (1, -1, -1, 1, 1)] & (1, -1, 0, -1, 1, 1) & (0, 0, 1, -1, 0, 0) & (-6, -5, 1, 1)  \\   
 125 & [1, (1, -1, 1, -1, 1)] & (1, -1, 0, 1, -1, 1) & (0, 0, 1, -1, 0, 0) & (2, -3, -1, 1) \\
 144 & [1, (1, 0, -1, 0, 1)] & (1, 0, 0, -1, 0, 1) & (0, 0, 1, -1, 0, 0) & (-4, -4, 1, 1) \\
 189 & [1, (1, -1, 0, 2, 1)] & (1, -1, 0, 0, 2, 1) & (0, 0, 1, -1, 0, 0) & (-5, -6, 0, 1) \\
 225 & [2, (4, 6, 5, 3, 1)] & (1, 3, 0, 5, 3, 1) & (0, 0, 1, -2, 0, 0) & (2, 1, -5, 2) \\
 229 & [1, (1, 0, 0, -1, 1)] & (1, 0, 0, 0, -1, 1) & (0, 0, 1, -1, 0, 0) & (-1, -4, 0, 1) \\
 256 & [1, (1, 0, 0, 0, 1)] & (1, 0, 0, 0, 0, 1) & (0, 0, 1, -1, 0, 0) & (0, -4, 0, 1) \\
 257 & [1, (1, 0, 1, -1, 1)] & (1, 0, 0, 1, -1, 1) & (0, 0, 1, -1, 0, 0) & (3, -4, -1, 1) \\
 272 & [1, (1, 0, 1, -2, 1)] & (1, 0, 0, 1, -2, 1) & (0, 0, 1, -1, 0, 0) & (0, -4, -1, 1) \\
 320 & [1, (1, -2, 0, 0, 2)] & (1, -2, 0, 0, 0, 2) & (0, 0, 1, -1, 0, 0) & (-8, -8, 0, 1) \\  
 333 & [1, (1, -1, -2, 0, 3)] & (1, -1, 0, -2, 0, 3) & (0, 0, 1, -1, 0, 0) & (-27, -12, 2, 1) \\
 392 & [1, (1, -1, 0, 1, 1)] & (1, -1, 0, 0, 1, 1) & (0, 0, 1, -1, 0, 0) & (-2, -5, 0, 1) \\
 400 & [1, (1, 0, 3, 0, 1)] & (1, 0, 0, 3, 0, 1) & (0, 0, 1, -1, 0, 0) & (12, -4, -3, 1) \\
 432 & [1,(1,0,-3,0,3)] & (1,0,0,-3,0,3) & (0,0,1,-1,0,0) & (-36,-12,3,1) \\
 441 & [2,(4,-2,-1,-1,1)] & (1,-1,0,-1,-1,1) & (0,0,1,-2,0,0) & (-6,-7,1,2) \\
 512 & [1,(1,0,-2,0,2)] & (1,0,0,-2,0,2) & (0,0,1,-1,0,0) & (-16,-8,2,1) \\
 513 & [2,(4,2,-3,-1,1)] & (1,1,0,-3,-1,1) & (0,0,1,-2,0,0) & (-14,-9,3,2) \\
 549 & [1,(1,-2,-2,3,3)] & (1,-2,0,-2,3,3) & (0,0,1,-1,0,0) & (-45,-18,2,1) \\ 
 576 & \ & (1, 0, 1, 0, 0, 1) & (0, 0, 2, -1, 0, 2) & (0, -3, -4, 4) \\
 576 & \ & (1, 0, 1, -1, 0, -1) & (0, 0, 2, -1, 0, -2)  & (5, 17, 16, 4) \\
 592 & [(1,(1,0,2,-2,1)] & (1,0,0,2,-2,1) & (0,0,1,-1,0,0) & (4,-4,-2,1) \\
 605 & [1,(1,-1,1,1,1)] & (1,-1,0,1,1,1) & (0,0,1,-1,0,0) & (2,-5,-1,1) \\
 656 & [1,(1,-2,-1,2,2)] & (1,-2,0,-1,2,2) & (0,0,1,-1,0,0) & (-20,-12,1,1) \\
 657 & [2,(4,-2,5,-1,1)] & (1,-1,0,5,-1,1) & (0,0,1,-2,0,0) & (18,-7,-5,2) \\
 697 & [1,(1,-1,2,-1,2)] & (1,-1,0,2,-1,2) & (0,0,1,-1,0,0) & (13,-7,-2,1) \\
 725 & [1,(1,-1,-3,1,1)] & (1,-1,0,-3,1,1) & (0,0,1,-1,0,0) & (-14,-5,3,1) \\
 761 & [1,(1,-1,1,2,1)] & (1,-1,0,1,2,1) & (0,0,1,-1,0,0) & (-1,-6,-1,1) \\
 784 & [2,(4,0,-3,0,1)] & (1,0,0,-3,0,1) & (0,0,1,-2,0,0) & (-12,-8,3,2) \\
 788 & [1,(1,-1,2,-2,2)] & (1,-1,0,2,-2,2) & (0,0,1,-1,0,0) & (10,-6,-2,1) \\
 832 & [1,(1,-2,0,-2,5)] & (1,-2,0,0,-2,5) & (0,0,1,-1,0,0) & (-24,-16,0,1) \\
 837 & [1,(1,-1,-3,-1,7)] & (1,-1,0,-3,-1,7) & (0,0,1,-1,0,0) & (-92,-27,3,1) \\
 873 & [2,(4,-6,-1,3,1)] & (1,-3,0,-1,3,1) & (0,0,1,-2,0,0) & (-22,-17,1,2) \\
 892 & [1,(1,-1,-1,0,2)] & (1,-1,0,-1,0,2) & (0,0,1,-1,0,0) & (-10,-8,1,1) \\
 981 & [1,(1,-1,-4,4,7)] & (1,-1,0,-4,4,7) & (0,0,1,-1,0,0) & (-135,-32,4,1) \\
 985 & [1,(1,-1,2,-3,2)] & (1,-1,0,2,-3,2) & (0,0,1,-1,0,0) & (5,-5,-2,1) \\
 1008 & [1,(1,0,-5,0,7)] & (1,0,0,-5,0,7) & (0,0,1,-1,0,0) & (-140,-28,5,1) \\
 1008 & [1,(1,0,5,0,7)] & (1,0,0,5,0,7) & (0,0,1,-1,0,0) & (140,-28,-5,1) \\
 1016 & [1,(1,-1,1,-2,2)] & (1,-1,0,1,-2,2) & (0,0,1,-1,0,0) & (2,-6,-1,1) \\
 1025 & [2,(4,-2,3,-1,1)] & (1,-1,0,3,-1,1) & (0,0,1,-2,0,0) & (10,-7,-3,2) \\
 1040 & [1,(1,0,-3,-2,5)] & (1,0,0,-3,-2,5) & (0,0,1,-1,0,0) & (-64,-20,3,1) \\
 1040 & [4,(16,32,25,8,1)] & (1,8,0,25,8,1) & (0,0,1,-4,0,0) & (-28,48,-25,4) \\
 1076 & [1,(1,-1,3,-3,2)] & (1,-1,0,3,-3,2) & (0,0,1,-1,0,0) & (13,-5,-3,1) \\
 1088 & [1,(1,-2,1,-2,3)] & (1,-2,0,1,-2,3) & (0,0,1,-1,0,0) & (-4,-8,-1,1) \\
 1088 & [1,(1,-2,5,-4,2)] & (1,-2,0,5,-4,2) & (0,0,1,-1,0,0) & (16,0,-5,1) \\
 1089 & [6,(36,66,43,11,1)] & (1,11,0,43,11,1) & (0,0,1,-6,0,0) & (-70,97,-43,6) \\
 1125 & [1,(1,-1,-4,4,1)] & (1,-1,0,-4,4,1) & (0,0,1,-1,0,0) & (-33,-8,4,1) \\
 1129 & [1,(1,-1,0,-1,2)] & (1,-1,0,0,-1,2) & (0,0,1,-1,0,0) & (-3,-7,0,1) \\
 1161 & [4,(112,116,51,11,1)] & (7,29,0,51,11,1) & (0,0,1,-4,0,0) & (-260,207,-51,4) \\
 1168 & [3,(9,24,23,8,1)] & (1,8,0,23,8,1) & (0,0,1,-3,0,0) & (-36,52,-23,3) \\
 1197 & [1,(1,-2,-4,5,7)] & (1,-2,0,-4,5,7) & (0,0,1,-1,0,0) & (-165,-38,4,1) \\
 1197 & [3,(9,3,-5,-1,1)] & (1,1,0,-5,-1,1) & (0,0,1,-3,0,0) & (-22,-13,5,3) \\
 1225 & [6,(36,42,23,7,1)] & (1,7,0,23,7,1) & (0,0,1,-6,0,0) & (-6,25,-23,6) \\
 1229 & [1,(1,-1,3,-1,3)] & (1,-1,0,3,-1,3) & (0,0,1,-1,0,0) & (32,-11,-3,1) \\
 1257 & [1,(1,0,-1,-1,2)] & (1,0,0,-1,-1,2) & (0,0,1,-1,0,0) & (-9,-8,1,1) \\
 1264 & [1,(1,0,3,-2,1)] & (1,0,0,3,-2,1) & (0,0,1,-1,0,0) & (8,-4,-3,1) \\
 1280 & [1,(1,0,-4,0,5)] & (1,0,0,-4,0,5) & (0,0,1,-1,0,0) & (-80,-20,4,1) \\
 1372 & [1,(1,-1,3,-4,2)] & (1,-1,0,3,-4,2) & (0,0,1,-1,0,0) & (6,-4,-3,1) \\
 1384 & [1,(1,-1,-1,2,2)] & (1,-1,0,-1,2,2) & (0,0,1,-1,0,0) & (-14,-10,1,1) \\
 1396 & [1,(1,-1,1,-1,2)] & (1,-1,0,1,-1,2) & (0,0,1,-1,0,0) & (5,-7,-1,1) \\
 1413 & [3,(9,-3,7,-1,1)] & (1,-1,0,7,-1,1) & (0,0,1,-3,0,0) & (26,-11,-7,3) \\
 1421 & [1,(1,-2,2,-1,2)] & (1,-2,0,2,-1,2) & (0,0,1,-1,0,0) & (7,-6,-2,1) \\
 1424 &    \    & (1, 0, 0, 1, -1, -1), &   (0, 0, 2, -1, 1, 3) & (-5, 18, -17, 4) \\ 
 1436 & [1,(1,-1,3,0,2)] & (1,-1,0,3,0,2) & (0,0,1,-1,0,0) & (22,-8,-3,1) \\
 1489 & [1,(1,-1,4,-1,2)] & (1,-1,0,4,-1,2) & (0,0,1,-1,0,0) & (29,-7,-4,1) \\
 1492 & [1,(1,-1,2,0,2)] & (1,-1,0,2,0,2) & (0,0,1,-1,0,0) & (14,-8,-2,1) \\
 1509 & [1,(1,-1,2,2,1)] & (1,-1,0,2,2,1) & (0,0,1,-1,0,0) & (3,-6,-2,1) \\
 1521 & [12,(6192,2748,463,35,1)] & (43,229,0,463,35,1) & (0,0,1,-12,0,0) & (-25480,5951,-463,12) \\
 1525 & [1,(1,-2,6,-5,5)] & (1,-2,0,6,-5,5) & (0,0,1,-1,0,0) & (75,-10,-6,1) \\
 1552 & [1,(1,-2,-3,4,5)] & (1,-2,0,-3,4,5) & (0,0,1,-1,0,0) & (-96,-28,3,1) \\
 1556 & [1,(1,-1,-1,1,2)] & (1,-1,0,-1,1,2) & (0,0,1,-1,0,0) & (-11,-9,1,1) \\
 1568 & [1,(1,0,-1,0,2)] & (1,0,0,-1,0,2) & (0,0,1,-1,0,0) & (-8,-8,1,1) \\
 1593 & [1,(1,-2,3,-1,2)] & (1,-2,0,3,-1,2) & (0,0,1,-1,0,0) & (15,-6,-3,1) \\
\end{longtable}
\end{center}

\end{document}